\documentclass[pdflatex,sn-nature]{sn-jnl}

\usepackage{graphicx}%
\usepackage{multirow}%
\usepackage{amsmath,amssymb,amsfonts}%
\usepackage{amsthm}%
\usepackage{mathrsfs}%
\usepackage[title]{appendix}%
\usepackage{xcolor}%
\usepackage{textcomp}%
\usepackage{manyfoot}%
\usepackage{booktabs}%
\usepackage{algorithm}%
\usepackage{algorithmicx}%
\usepackage{algpseudocode}%
\usepackage{listings}%

\theoremstyle{thmstyleone}%
\newtheorem{theorem}{Theorem}[section]
\newtheorem{lemma}{Lemma}[section]
\newtheorem{corollary}{Corollary}[theorem]

\theoremstyle{thmstyletwo}%

\theoremstyle{thmstylethree}%
\newtheorem{assumption}{Assumption}[section]%
\newtheorem{example}{Example}[section]%

\newcommand{\rb}[1]{ \left( #1 \right) }
\newcommand{\rrb}[1]{ \left\lbrace #1 \right\rbrace }
\newcommand{\curbr}[1]{ \lbrace #1 \rbrace }

\newcommand{\abs}[1]{\left|#1\right|}
\newcommand{\floorfct}[1]{\left\lfloor #1 \right\rfloor}
\newcommand{\norm}[1]{\left\lVert#1\right\rVert}
\newcommand{\snorm}[1]{\lVert#1\rVert}

\newcommand{\suml}{\sum\limits}
\newcommand{\prodl}{\prod\limits}
\newcommand{\intl}{\int\limits}
\newcommand{\liml}{\lim\limits}
\newcommand{\supl}{\sup\limits}

\newcommand{\LandauO}{ \mathcal{O} }

\newcommand{\IndNr}[1]{\mathds{1}_{\left\lbrace#1\right\rbrace}}

\newcommand{\BigS}{\operatorname{S}}

\newcommand{\C}{\mathbb{C}}

\newcommand{\Kbb}{\mathbb{K}}
\newcommand{\Lbb}{\mathbb{L}}
\newcommand{\N}{\mathbb{N}}

\newcommand{\R}{\mathbb{R}}
\newcommand{\Tbb}{\mathbb{T}}
\newcommand{\Z}{\mathbb{Z}}

\usepackage{dsfont}
\newcommand{\ONE}{\mathds{1}}

\newcommand{\Beta}{\operatorname{B}}

\newcommand{\Tau}{\operatorname{T}}

\begin{document}

\title[Article Title]{Deconvolution of distribution functions without integral transforms}

\author*[1]{\fnm{Henrik} \sur{Kaiser}}\email{henrik.kaiser@plus.ac.at}

\affil*[1]{\orgdiv{Dept. of Artificial Intelligence and Human Interfaces}, \orgname{Paris Lodron University}, \orgaddress{\street{Hellbrunner Strasse 34}, \city{Salzburg}, \postcode{5020}, \country{Austria}}}

\abstract{We study the recovery of the distribution function $F_X$ of a random variable $X$ that is subject to an independent additive random error $\varepsilon$. To be precise, it is assumed that the target variable $X$ is available only in the form of a blurred surrogate $Y = X + \varepsilon$. The distribution function $F_Y$ then corresponds to the convolution of $F_X$ and $F_\varepsilon$, so that the reconstruction of $F_X$ is some kind of deconvolution problem. Those have a long history in mathematics and various approaches have been proposed in the past. Most of them use integral transforms or matrix algorithms. The present article avoids these tools and is entirely confined to the domain of distribution functions. Our main idea relies on a transformation of a first kind to a second kind integral equation. Thereof, starting with a right-lateral discrete target and error variable, a representation for $F_X$ in terms of available quantities is obtained, which facilitates the unbiased estimation through a $Y$-sample. It turns out that these results even extend to cases in which $X$ is not discrete. Finally, in a general setup, our approach gives rise to an approximation for $F_X$ as a certain Neumann sum. The properties of this sum are briefly examined theoretically and visually. The paper is concluded with a short discussion of operator theoretical aspects and an outlook on further research. Various plots underline our results and illustrate the capabilities of our functions with regard to estimation.}



\maketitle

\section{Introduction}

Let $F_X$ be the distribution function (d.f. or d.fs., for short) of a random variable $X$, and let $X_1, \hdots, X_n \sim F_X$ be an associated sample of independent, identically distributed (i.i.d.) observations, for $n\in\N$. Then, a well-known consistent estimator for $F_X$ is the \textit{empirical distribution function} (abbr.: e.d.f. or e.d.fs.)
\begin{align} \label{MeasErr1}
F_X(\xi, n) := \frac{1}{n} \suml_{k=1}^n \ONE_{ \curbr{ X_k \leq \xi } } \hspace{1cm} (\xi\in\R).
\end{align}
Often, however, $X$ is unobservable. Such a situation is the subject of the present article, assuming accessibility of $X$ only through a surrogate variable $Y$, which differs from $X$ by an additive independent random \textit{error} or \textit{noise} $\varepsilon$. Formally,
\begin{align} \label{ZVFaltung}
Y = X+\varepsilon,
\end{align}
with independent $X$ and $\varepsilon$. The above setting is known as the \textit{additive model of errors in variables}. The variable $Y$ can be conceived in multiple ways, e.g., as an an imprecise or blurred measurement, due to random effects. Thus, deconvolution is relevant in many fields, such as medicine and econometry \cite[see][]{Sarkar2014, WrightEtAl2014, Bonhomme2010, GonzalezEtAl2016}. In the described situation, the d.f. of $Y$, written $F_Y$, is represented by the \textit{additive convolution} of $F_X$ with the d.f. $F_\varepsilon$ of $\varepsilon$, that is
\begin{align} \label{Verteilungsfaltung}
F_Y(\xi) = \intl_{-\infty}^\infty F_X(\xi-z)F_\varepsilon(dz) \hspace{1cm} (\xi \in \R).
\end{align}
Here, we simply speak of convolution without a prefix, since there will be no danger of confusion with other kinds of convolution. Theoretically, the d.fs. $F_Y$ and $F_\varepsilon$ are both supposed to be completely known, but in practice at least a $Y$-sample will serve to estimate $F_Y$ by virtue of its empirical analogue. With regard to $F_\varepsilon$, various assumptions are common. Specifically if $F_X$ is absolutely continuous with density $f_X$, the d.f. $F_Y$ has the density
\begin{align} \label{Dichtenfaltung}
f_Y(\xi) = \intl_{-\infty}^\infty f_X(\xi-z) F_\varepsilon(dz) \hspace{1cm} (\xi\in\R).
\end{align}
However, this text focusses on d.fs. rather than densities, since we consider the assumption of the actual existence of a target density as too restrictive. Moreover, density estimation already bears various difficulties in case of unblurred observations. In fact, \cite{parzen1962, Rosenblatt1956} were able to show that an unbiased estimator does then not exist. The above kind of integrals in older texts are rather known as \textit{(Stieltjes) resultant}, and some authors even use the German word \textit{Faltung} (see, e.g., \cite[][p. 51--52]{titchmarsh1937} or \cite[][p. 84]{widder1946}). The inverse convolution, i.e., the reconstruction of $F_X$ or $f_X$, is called \textit{deconvolution} and amounts to solving an integral equation of the first kind. Due to the complicated structure of convolution products, it can be a nasty problem. An historical overview on convolution is provided by \cite{Dominguez2015}. For a general introduction to measurement errors in statistics, the reader may consult \cite{Meister2009, Yietal2021}.
\\
\hspace*{1em}Chapters on deconvolution can be found in many textbooks on analysis, with a main focus on Fourier analysis (cf. \cite[Ch. XI]{titchmarsh1937}, \cite[][Ch. V, $\S$8]{widder1946}, \cite[][$\S$1.9]{tricomi1985integral} and \cite{HirschmanWidder1955}). The reason is that convolution corresponds to the multiplication of Fourier transforms. To be more precise, if we denote by
\begin{align} \label{DefcFX}
\Phi_X(t) := \intl_{-\infty}^\infty e^{itx} F_X(dx) \hspace{1cm} (t\in\R)
\end{align}
the characteristic function of $X$ (c.f. or c.fs., for brevity), i.e., the \textit{Fourier-Stieltjes transform of $F_X$}, then the convolution equation (\ref{Verteilungsfaltung}) is equivalent to
\begin{align} \label{additFiVcF}
\Phi_Y = \Phi_X\Phi_\varepsilon.
\end{align}
Notice that a consideration of convolution in the Fourier domain is not a restriction, as c.fs. exist for any distribution. Moreover, all d.fs. uniquely can be identified by their c.fs. and even reconstructed via inversion formulae. Equation (\ref{additFiVcF}) immediately warns us that $\Phi_X$ is only identifiable if $\Phi_\varepsilon = 0$ on a set of Lebesgue measure zero. Then, $\Phi_X = \curbr{\Phi_\varepsilon}^{-1} \Phi_Y$ (within the zero set, this holds by continuity), and inversion directly yields $F_X$ (or $f_X$, if existing). The fact that $\Phi_X$ may not be identifiable leads to unboundedness of the inverse operator, when considering convolution on function spaces. Accordingly, following the characterization of well-posedness due to \cite{Hadamard1902}, deconvolution is often classified an \textit{ill-posed inverse problem}.
\\
\hspace*{1em}Deconvolution becomes even more challenging if $\Phi_Y$ is only estimable. The article \cite{StefCarr1988} can be considered the first contribution to this area. Due to the independence of $X$ and $\varepsilon$, any sample $Y_1, \hdots, Y_n \sim F_Y$ unrestrictedly can be assumed i.i.d., so that the \textit{empirical characteristic function} (abbr.: e.c.f. or e.c.fs.) of $Y$, viz
\begin{align*}
\Phi_Y(t,n)  :=&\, \intl_{-\infty}^\infty e^{ity} F_Y(dy, n) \hspace{1cm} (t \in \R),
\end{align*}
may serve as a straightforward consistent plug-in estimator for $\Phi_Y$. Under the assumption of a known error distribution, in view of (\ref{additFiVcF}), in \cite{StefCarr1988}, it was proposed to estimate $\Phi_X$ through $\curbr{\Phi_\varepsilon}^{-1}\Phi_Y(\cdot, n)$ and eventually also the density $f_X$ by means of a Fourier inversion formula and a suitable smoothing kernel. A corresponding estimator for the d.f. thereof can be obtained by integration \cite[see][]{Meister2009}. These kernel estimators are biased with respect to the target and share similar properties, which substantially vary with the predetermined error distribution and the choice of the kernel. Basically, the approach of \cite{StefCarr1988} resembles the idea of \cite{parzen1962, Rosenblatt1956}, in case of unblurred observations. Additional difficulties, however, arise from the fact that the underlying estimator $\curbr{\Phi_\varepsilon}^{-1}\Phi_Y(\cdot, n)$ is often unbounded, due to the decay of $\Phi_\varepsilon$. Nevertheless, kernel estimators gained a wide popularity in literature. Earliest asymptotic results go back to \cite{Fan1991a, Fan1991b} and are still relevant in recent literature \cite[cf.][]{Adusumilli2020, Thuy2023, Zhu2022, DattnerEtAl2011, Hesse2006}. Attempts to construct kernel estimators that overcome classical problems usually result in strong assumptions on the associated distributions; see \cite{BelomGolden2020a, GoldenshlugerKim2021, Hall2007} and \cite[][$\S$2.2.3]{Meister2009}. As alternatives to the predominant Fourier methods, we mention wavelet-based estimators (cf. \cite{Cao2023} and \cite[][$\S$2.2.2]{Meister2009}) and maximum likelihood methods \cite{Guan2021}.
\\
\hspace*{1em}The assumption of a known error distribution is a common starting point, to become familiar with the problem of errors in variables, although maybe unrealistic in practice. If $F_\varepsilon$ is not completely known, additional data is required to facilitate a characterization of this d.f. and to keep $F_X$ identifiable. Various techniques have been developed in literature, assuming the availability of information on $F_\varepsilon$ in different ways; cf. \cite{Adusumilli2020, Delaigle2016, Kato2018} and \cite[][$\S$2.6]{Meister2009}.
\\
\hspace*{1em}In the additive model of errors in variables, if $\Tbb_X$ and $\Tbb_\varepsilon$ denote the supports of the indicated variables, the probability mass function (p.m.f. or p.m.fs., for short) of $F_Y$ is determined by those of $F_X$ and $F_\varepsilon$ through
\begin{align} \label{MassFaltung}
F_Y\curbr{y} = \suml_{ \substack{ (x, z) \in \Tbb_X \times \Tbb_\varepsilon \\ x+z = y } } F_X\curbr{x} F_\varepsilon\curbr{z} \hspace{1cm} (y \in \R).
\end{align}
This convolution equation becomes particularly interesting if $F_X$ and $F_\varepsilon$ are both purely discrete, with left-bounded $\Tbb_X$ and $\Tbb_\varepsilon$. Then, $F_X$ is uniquely determined by its atoms, which can be recovered from those of $F_Y$ and $F_\varepsilon$ as the solution of a linear equation system. For instance, assuming that $\Tbb_X = \Tbb_\varepsilon = \N_0$, it is clear that also $\Tbb_Y = \N_0$. Hence, for an arbitrary $K \in \N$, the convolution equation (\ref{MassFaltung}) implies that $F_Y\curbr{k} = \sum_{z=0}^k F_X\curbr{k-z} F_\varepsilon\curbr{z}$, for all $0 \leq k \leq K$. In terms of the matrix
\begin{align*}
M_K :=
\begin{pmatrix}
F_\varepsilon\curbr{0} & F_\varepsilon\curbr{1} & \hdots & F_\varepsilon\curbr{K} \\
0 & F_\varepsilon\curbr{0} & \hdots & F_\varepsilon\curbr{K-1} \\
\vdots & \vdots & \ddots & \vdots \\
0  & 0 & \hdots & F_\varepsilon\curbr{0}
\end{pmatrix}
\in \R^{(K+1) \times (K+1)}
,
\end{align*}
this is equivalent to
\begin{align*}
\begin{pmatrix}
F_Y\curbr{K} \\
\vdots \\
F_Y\curbr{0}
\end{pmatrix}
=
M_K
\begin{pmatrix}
F_X\curbr{K} \\
\vdots \\
F_X\curbr{0}
\end{pmatrix}
.
\end{align*}
Therefore, if $M_K$ is invertible, i.e., $\operatorname{det} M_K > 0$ or simply $F_\varepsilon\curbr{0}>0$, one directly obtains a representation for the vector $( F_X\curbr{K}, \hdots, F_X\curbr{0} )$ in terms of $F_Y$ and $F_\varepsilon$. This technique requires a tremendous computational effort as $K$ increases. \textit{Discrete deconvolution}, as it is called, has been discussed in many fields, however, apparently not so often with a probabilistic background. Most contributions aim for efficient inversion algorithms. Again, Fourier methods are of frequent use, specifically the discrete Fourier transform, i.e., the Fourier-Stieltjes transform of a discrete measure \cite[cf.][]{SilvermanPearson1973, HallQiu2005, deHoog2018}. On the other hand, some authors tackle the problem directly in the matrix domain \cite[see][]{HuckleSed2013, ZhigljavskyEtAl2016}.
\\
\hspace*{1em}The interest in deconvolution during the last years, at least with a probabilistic background, appears to be on a constant ordinary level, with a lack of novel ideas. In particular, to the best of our knowledge, there still is no generally feasible technique to estimate the d.f. of an arbitrary random quantity $X$ that is subject to additive measurement errors. Moreover, even in setups where the existence of a non-parametric unbiased estimator for $F_X$ is obvious, such an estimator has not yet been established. These considerations motivated this work. As the title suggests, unlike most authors, we do not use integral transforms and instead completely conduct all of our research in the domain of d.fs., starting with a transformation of the respective convolution equation to an integral equation of the second kind. For a better insight on the challenges in deconvolution, our first study is dedicated to simpler setups. To be precise, we first derive a general formula for discrete deconvolution, that we eventually apply to various probabilistic scenarios. Our initial idea relies on the observation that the $k$-th jump point of $F_Y$, in case of non-negative integer-valued $X$ and $\varepsilon$, merely depends on the first $k$ consecutive atoms of $F_X$ and $F_\varepsilon$, from which conversely a recursion for the p.m.f. of $X$ can be obtained. The pattern behind this recursion unfolds through a technique that originates in the theory of integral equations, and gives rise to a finite representation for the p.m.f. of $X$, solely in terms of quantities that depend on $Y$ and $\varepsilon$. Finally, the transition to d.fs. is straightforward, and an unbiased non-parametric estimator for $F_X$ can be established. Subsequently, we vary our assumptions on $X$ and $\varepsilon$, before we eventually turn to arbitrary setups. In these, we are able to approximate $F_X$ or even $f_X$ through a Neumann sum, whose convergence is examined in a special case, followed by a simulation study for normally distributed errors. Also, a short discussion of the invertibility of our convolution operator is included. We finish our study with a glimpse into the Fourier domain and an outlook on future goals, aiming to develop an unrestricted deconvolution estimator.

\section{Notation and preliminaries} \label{SecNotNPrel}

For an arbitrary function $Q : \R \rightarrow \R$, without loss of generality, we write $Q(\xi-)$ and $Q(\xi+)$, respectively, for the limit from the left and from the right at $\xi \in \R$, with $Q\curbr{\xi} := Q(\xi+)- Q(\xi-)$. The $Q$-atoms, i.e., discontinuities of $Q$, are the set $D_Q := \curbr{ \xi \in \R : Q\curbr{\xi} \neq 0 }$, and $C_Q := \R \setminus D_Q$ are the associated continuity points/intervals. Furthermore, $Q(\pm\infty) := \lim_{\xi \rightarrow \pm\infty} Q(\xi)$, whenever one of the indicated limits exists. If both exist and $Q$ is continuous on $\R$, we say that it is continuous on $\overline \R := \R \cup \curbr{\pm \infty}$. In addition, $Q$ is \textit{right-lateral} (\textit{left-lateral}) if there exists $\xi_0 \in \R$, with $Q(\xi) = 0$, for all $\xi < \xi_0$ ($\xi > \xi_0$). As usual, $\ONE_\mathcal{M}$ refers to the indicator of the set $\mathcal{M} \subset \overline \R$, and $\snorm{ Q }_p$, for $0 < p \leq \infty$, stands for the $L^p$-norm on $\R$ (with respect to the Lebesgue measure). Also, $\ell^1(\Kbb) := \curbr{ ( a(k) )_{k \in \Kbb} \subset \C : \sum_{k \in \Kbb} |a(k)| < \infty }$, for $\Kbb \subseteq \Z$. Letting $\mathcal{B}(\R)$ be the Borel $\sigma$-algebra on $\R$, the set $A +x B := \curbr{ a+xb : (a, b) \in A\times B }$ refers to the Minkowski linear combination of $A, B \in \mathcal{B}(\R)$, with fixed $x \in \R \setminus \curbr{0}$. The big $\LandauO$ and small $o$ have their usual meaning, and we indicate by $i$, $\Re z$, $\Im z$ and $\overline z$, respectively, the imaginary unit, the real part, the imaginary part and the complex conjugate of $z \in \C$. Finally, we agree that empty sums equal zero and empty products are equal to one.
\\
\hspace*{1em}A \textit{signed measure} $\mu : \mathcal{B}(\R) \rightarrow \overline\R$ is a countably additive mapping, i.e., $\mu(\bigcup_{j=1}^\infty A_j) = \sum_{j=1}^\infty \mu(A_j)$, for every disjoint sequence $(A_j)_{j\in \N} \subset \mathcal{B}(\R)$. If even $\mu \geq 0$, then $\mu$ simply is a \textit{measure}. Moreover, sometimes, we use the notion of a \textit{complex measure}, by which we mean compositions of the form $\mu := \mu_1 + i\mu_2$, for signed measures $\mu_1, \mu_2 : \mathcal{B}(\R) \rightarrow \overline \R$. In any case, if $|\mu(A)| < \infty$, for all $A \in \mathcal{B}(\R)$, we add the prefix \textit{finite}. The finite signed ($\Kbb = \R$) and complex ($\Kbb = \C$) measures on $(\Kbb, \mathcal{B}(\R))$, respectively, form the vector spaces $\mathcal{M}(\Kbb, \mathcal{B}(\R))$. Clearly, $\mathcal{M}(\R, \mathcal{B}(\R)) \subset \mathcal{M}(\C, \mathcal{B}(\R))$. Finally, $\mu$ is a \textit{probability measure}, if $0 \leq \mu \leq 1$. Specifically $\delta_{ \curbr{x} }$ indicates the Dirac measure with mass at $x \in \R$. Now, the support $\Tbb_\mu$ of a complex measure $\mu$ is characterized by the property that $\mu(A) = \mu(A \cap \Tbb_\mu)$, for each $A \in \mathcal{B}(\R)$, and the the \textit{total variation} on $A \in \mathcal{B}(\R)$ \cite[see][$\S$9A]{axler2019measure} is defined as
\begin{align*}
|\mu|(A) := \sup\rrb{ \suml_{k=1}^K |\mu(A_k)| : K \in \N, ~ A_1, \hdots, A_K \in \mathcal{B}(\R) \operatorname{disjoint, with} \bigcup\limits_{k=1}^K A_k \subset A }.
\end{align*}
It is known that $|\mu|(A)<\infty$, for all $A \in \mathcal{B}(\R)$ and $\mu \in \mathcal{M}(\C, \mathcal{B}(\R))$, i.e., all finite complex measures are of \textit{finite total variation} on each Borel set. In particular, the mapping $| \mu | : \mathcal{B}(\R) \rightarrow [0, \infty]$ generally is a measure, and even finite, for all $\mu \in \mathcal{M}(\C, \mathcal{B}(\R))$. The d.f. induced by a complex measure $\mu : \mathcal{B}(\R) \rightarrow \overline \R + i \overline \R$ on $\R$ is denoted by $F_\mu(\xi) := \mu( (-\infty, \xi] )$, for $\xi \in \R$. If $\mu \in \mathcal{M}(\C, \mathcal{B}(\R))$, the limits $F_\mu(\pm\infty)$ exist and hence $F_\mu \in L^\infty(\R)$, i.e., using the terminology of older texts, $F_\mu : \R \rightarrow \C$ is of \textit{bounded variation} on $\overline \R$ \cite[see][$\S$2.1]{wheeden2015measure}. In particular, $\mu(dx) = F_\mu(dx)$. For that reason, following the convention for functions of bounded variation, we may also refer to Lebesgue integrals with respect to complex measures as Stieltjes integrals, and vice versa, and to $\Phi_\mu(t) := \int_{-\infty}^\infty e^{itx} F_\mu(dx)$ as the \textit{Fourier-Stieltjes transform}. The latter represents a complex-valued uniformly continuous function of $t \in \R$, for each $\mu \in \mathcal{M}(\C, \mathcal{B}(\R))$. Finally, any sequence $( p(z) )_{z \in \Z} \subset \C$ can be identified with a discrete complex measure of the form $\sum_{z\in\Z} p(z) \delta_{ \curbr{z} }$, however, which is possibly of infinite total variation on unbounded sets. For the associated d.f., we write
\begin{align*}
\Theta\curbr{p}(\xi) := \suml_{ z \in \Z } p(z) \IndNr{ \xi \geq z } = \suml_{ z = -\infty }^{ \floorfct{\xi} } p(z) \hspace{1cm} (\xi \in \R).
\end{align*}
In the sequel, all sequences are right-lateral, so that the sum is always finite, yet, possibly divergent as $\xi \rightarrow \infty$.
\\
\hspace*{1em}Every probability measure corresponds to a random variable. Conversely, for a random variable $B$, we indicate by $\mu_B$, $\Tbb_B$, $F_B$ and $\Phi_B$, respectively, the associated probability measure, its support, d.f. and Fourier-Stieltjes transform, for which we use the more common terminology of a c.f..  Empirical analogues and possibly existing density are denoted by $F_B(\cdot, n)$, $\Phi_B(\cdot, n)$ and $f_B$, respectively. Furthermore, $F_B\curbr{x}$, for $x \in \R$, refers to the p.m.f.. The distribution of $B$ is \textit{right-lateral} (\textit{left-lateral}) if and only if $\Tbb_B$ is bounded to the left (right), and otherwise it is \textit{bilateral}. Any c.f. satisfies $\Phi_B(0) = 1$ and $0 \leq \abs{ \Phi_B } \leq 1$, with complex conjugate $\overline{ \Phi_B(t)} = \Phi_B(-t)$, for all $t \in \R$. Moreover, the c.f. reflects the kind of distribution. On the one hand, $\Phi_B$ is almost periodic in the sense of Bohr \cite[see][]{Bohr1932}, if and only if $\mu_B$ is discrete. This is equivalent to the existence of $L_\varepsilon > 0$, given an arbitrary $\varepsilon > 0$, such that each interval of length $L_\varepsilon$ contains a number $\tau_\varepsilon$ with $\abs{\Phi_B(t+\tau_\varepsilon)-\Phi_B(t)}\leq \varepsilon$, for all $t\in\R$. On the other hand, a necessary condition for $\Phi_B$ to vanish at infinity is continuity of $\mu_B$, particularly absolute continuity being sufficient.
\\
\hspace*{1em}Basically, there are two main types of convolution. Firstly, the convolution of complex measures $\mu, \nu : \mathcal{B}(\R) \rightarrow \overline \R + i\overline \R$ is given by the integral
\begin{align} \label{2025070901}
(\mu \ast \nu)(A) := \intl_\R \intl_\R \ONE_A(x+y) \nu(dx) \mu(dy) \hspace{1cm} (A \in \mathcal{B}(\R)),
\end{align}
which is finite, e.g., if $|\nu|( A - \Tbb_\mu ) < \infty$ and $|\mu|( A - \Tbb_\nu ) < \infty$, in which case even $|\mu\ast\nu|(A) < \infty$. Since $\abs{(\mu \ast \nu)(A)} \leq \abs{\mu}(\R)\abs{\nu}(\R)$, uniformly with respect to $A \in \mathcal{B}(\R)$, the convolution of $\mu, \nu \in \mathcal{M}(\C, \mathcal{B}(\R))$ is always well-defined, with $\mu \ast \nu \in \mathcal{M}(\C,\mathcal{B}(\R))$. Specifically $(F_\mu\ast F_\nu)(\xi) := (\mu \ast \nu)((-\infty, \xi])$ is the convolution of the corresponding d.fs. $F_\mu$ and $F_\nu$, referred to as the \textit{Stieltjes convolution} or \textit{Stieltjes resultant}, in the classical fashion. Moreover, if $\mu := \sum_{k=-\infty}^\infty p(k) \delta_{\curbr{k}}$ and $\nu := \sum_{k=-\infty}^\infty q(k) \delta_{\curbr{k}}$ are associated with two sequences $( p(\ell) )_{\ell \in \Z}, ( q(\ell) )_{\ell \in \Z} \subset \C$, then $(p \ast q)(\ell) := (\mu \ast \nu)(\curbr{\ell}) = \sum_{z = -\infty}^\infty p(\ell-z)q(z)$, for $\ell \in \Z$, is known as the \textit{discrete convolution}. It is always well-defined, whenever both sequences are either right- or left-lateral, resulting in a unilateral sequence again. In the right-lateral case, we have
\begin{align} \label{2025020603}
(\Theta\curbr{p}\ast \Theta\curbr{q})(x) = (p \ast \Theta\curbr{q})(\floorfct{x}) = \Theta\curbr{p \ast q}(x) \hspace{1cm} (x \in \R).
\end{align}
In particular, a repeated application of this identity shows that $\Theta\curbr{p}^{\ast j}(x) = \Theta\curbr{p^{\ast j}}(x)$, for all $(j, x) \in \N_0 \times \R$. The second main type of convolution is the \textit{$L^1$-convolution}, for $f, g \in L^1(\R)$, meaning the integral $(f \ast g)(x) := \int_\R f(x - y)g(dy)$, that is well-defined for Lebesgue almost all $x \in \R$ and fulfills $f \ast g \in L^1(\R)$. In the sequel, whenever the kind of convolution is clear, prefixes will be omitted. In each of the above cases, convolution commutes and therefore can be conceived as a kind of product. Regarding the convolution of complex measures, there even exists a neutral element, namely the Dirac measure $\delta_{\rrb{0}} \in \mathcal{M}(\R, \mathcal{B}(\R))$. On the other side, also $\int_\R f(x-y) \delta_{\curbr{0}}(dy) = f(x)$, for all $f \in L^1(\R)$ and Lebesgue almost every $x \in \R$. This mixture-type convolution integral, however, admits no equivalent representation in the sense of $L^1(\R)$, since $\delta_{\curbr{0}}$ is not absolutely continuous with respect to the Lebesgue measure. Thanks to the existence of a neutral element, we can eventually define convolution powers of $\mu : \mathcal{B}(\R) \rightarrow \overline \R + i\overline \R$ through $\mu^{\ast 0} := \delta_{ \curbr{0} }$ and
\begin{align*}
\mu^{\ast j}(A) := \intl_\R \hdots \intl_\R \ONE_{A}(x_1 + \hdots x_j) \mu(dx_1) \hdots \mu(dx_j) \hspace{1cm} ((j, A) \in \N \times \mathcal{B}(\R)),
\end{align*}
the integral being well-defined, with $|\mu^{\ast j}|(A) < \infty$, if $|\mu|(A-(j-1)\Tbb_\mu) < \infty$. If even $|\mu|(A-r\Tbb_\mu) < \infty$, for all $0 \leq r \leq j-1$, the $j$-th convolution power fulfills the recursion $\mu^{\ast j} = \mu \ast \mu^{\ast (j-1)}$. All this clearly holds for $\mu \in \mathcal{M}(\C, \mathcal{B}(\R))$. Various properties of convolution powers are verified in Appendix \ref{AppConvIdMeas}, among these the binomial convolution theorem.
\\
\hspace*{1em}Finally, for a norm $\snorm{ \cdot }_V$ on a vector space $V$, it is known that the norm of the linear operator $\Tau : V \rightarrow V$ is given by $\snorm{ \Tau } = \sup\curbr{ \snorm{ \Tau v }_V : v \in V, ~ \snorm{ v }_V=1 }$. If $\snorm{ \Tau } < \infty$, then $\Tau$ is called bounded, which is equivalent to continuity. We denote by $\mathcal{L}(V)$ the space of bounded linear operators on $V$, and specifically by $\operatorname{Id}_V$ the identity operator on $V$. Then, $\Tau \in \mathcal{L}(V)$ is invertible if and only if there exists $\widetilde \Tau \in \mathcal{L}(V)$ with $\widetilde \Tau \Tau = \operatorname{Id}_V = \Tau \widetilde \Tau$, in which case it is common to write $\Tau^{-1} := \widetilde \Tau$. In particular, $\Tau$ is then bijective, and the inverse also is a continuous, i.e., bounded, linear operator. Furthermore, a vector space $V$ is a Banach space, if it is complete with respect to $\snorm{\cdot}_V$. With $\Kbb \in \curbr{\R, \C}$ and $\snorm{\mu}_{TV} := \abs{\mu}(\R)$ denoting the total variation norm of $\mu \in \mathcal{M}(\Kbb, \mathcal{B}(\R))$, examples for Banach spaces are $(\mathcal{M}(\Kbb, \mathcal{B}(\R)), \snorm{\cdot}_{TV})$ and $(L^1(\R), \snorm{\cdot}_1)$. For further operator theoretical basics, we refer to \cite{axler2019measure, robinson2020introduction}.

\section{Deconvolution with a right-lateral discrete noise} \label{SomeSimpExamp}

We begin our examination with a discrete deconvolution problem, motivated by the aim to recover the d.f. $F_X$ in settings of errors in variables, in which both components $X$ and $\varepsilon$ are associated with right-lateral discrete distributions, and specifically $F_X$ distributes its mass on a monotonic set. As a particular consequence, we will eventually be able to represent the Dirac measure in terms of any right-lateral sequence. In the second part of this section, we will apply this result to a more general setting.

\subsection{A right-lateral discrete target with a monotonic support}

Throughout this paragraph, the following assumptions are supposed to hold.

\begin{assumption} \label{Ass2025081801}
The sequence $( r(\ell) )_{\ell\in \Z}$ is defined by $r(\ell) := (q\ast p(\ell, \cdot))(\ell)$, viz
\begin{align} \label{2025073101}
r(\ell) = \suml_{z = -\infty}^\infty q(\ell-z)p(\ell, z) \hspace{1cm} (\ell \in \Z),
\end{align}
for two sequences $( q(\ell) )_{\ell\in\Z}, ( p(\ell, z) )_{ (\ell, z) \in \Z^2 } \subset \C$, such that $q(\ell) = 0$, for all $\ell \in -\N$, and $p(\ell, z) = 0$, for all $(\ell, z) \in \Z \times -\N$.
\end{assumption}

Observe that $r$ is actually composed by a finite number of summands only, that is, $r(\ell) = \sum_{z=0}^\ell q(\ell-z)p(\ell, z)$, for all $\ell \in \Z$. Our first step consists in a transformation of this convolution equation. For this purpose, we define $\Lbb := \curbr{0, \hdots, L_0-1}$, in terms of
\begin{align*}
L_0 := \sup\rrb{ L \in \N_0 : p(\ell, 0) \neq 0 \mbox{ for all } 0 \leq \ell < L },
\end{align*}
assuming without loss of generality that $L_0 \geq 1$. Moreover, we introduce the sequences $( \ddot r(\ell) )_{\ell \in \Z}$ and $( \ddot p_+(\ell, z) )_{(\ell, z) \in \Z^2}$, given by
\begin{align} \label{2025073103}
\ddot r(\ell) :=&\, \frac{ r(\ell) }{ p(\ell, 0) } \ONE_\Lbb(\ell), \\ \label{2025073102}
\ddot p_+(\ell, z) :=&\, \rb{ \delta_{\curbr{0}}(\curbr{z}) - \frac{ p(\ell, z) }{ p(\ell, 0) } } \ONE_\Lbb(\ell).
\end{align}
Notice that $\ddot r(\ell) = \ddot p_+(\ell, z) = 0$, for all $(\ell, z) \in \Z \setminus \Lbb \times \Z$. In addition, $\ddot p_+(\ell, z) = 0$, even if $(\ell, z) \in \Lbb \times -\N_0$. Lastly, we define $( q|_\Lbb(\ell) )_{\ell \in \Z}$ as
\begin{align} \label{2025091101}
q|_\Lbb(\ell) := q(\ell) \ONE_\Lbb(\ell).
\end{align}
Now, with the aid of the above quantities, the equation (\ref{2025073101}) implies that
\begin{align} \label{2025072806}
q|_\Lbb(\ell) = \ddot r(\ell) + (q|_\Lbb \ast \ddot p_+(\ell, \cdot))(\ell) \hspace{1cm} (\ell \in \Z).
\end{align}
We remark that the convolution on the right hand side actually is a convolution of two discrete measures, of which the first has its atoms on $( q|_\Lbb (z) )_{z \in \Z}$, whereas the atoms of the second measure lie on $( \ddot p_+(\ell, z) )_{ z \in \Z }$, depending on $\ell \in \Z$. Accordingly, to follow the classical theory on integral equations, we identify the above as a Volterra-type equation of the second kind. Since the convolution product is the sum of the first $\ell+1$ consecutive atoms, for each $\ell \in \Lbb$, it is particularly convenient to characterize the associated solution.

\begin{lemma} \label{Lem2025072901}
Under Assumption \ref{Ass2025081801}, a function $a : \R \rightarrow \C$ satisfies $a(\ell) = q|_\Lbb(\ell)$, for all $\ell \in \Z$, if and only if
\begin{align} \label{2025072901}
a(\ell) = \ddot r(\ell) + (a \ast \ddot p_+(\ell, \cdot))(\ell) \hspace{1cm} (\ell \in \Z).
\end{align}
\end{lemma}

\begin{proof}
It is clear from (\ref{2025072806}) that $( q|_\Lbb(\ell) )_{\ell \in \Z}$ satisfies the indicated identity. For the proof of the other implication, suppose that $a : \R \rightarrow \C$ fulfills (\ref{2025072901}). Then, $a(\ell) = 0$, for all $\ell \in \Z \setminus \Lbb$. Moreover, with the aid of (\ref{2025073101}), for $\ell \in \Lbb$, we can write $a(\ell) = q(\ell) + ( p(\ell, 0) )^{-1} \sum_{z=1}^\ell (q(\ell-z) - a(\ell-z)) p(\ell, z)$. It shows that $a(0) = q(0)$, $a(1) = q(1)$ and so on. In summary, $a(\ell) = q(\ell)$, for each $\ell \in \Z$.
\end{proof}

A common technique to approximate a solution for second kind integral equations is the \textit{Picard iteration}, also known as the \textit{method of successive approximations} \cite[see, e.g.,][$\S$1.3]{tricomi1985integral}. We adopt this approach, letting $q(\ell, 0) := \ddot r(\ell)$ and
\begin{align} \label{2025072812}
q(\ell,m) := \ddot r(\ell) + (q(\cdot, m-1) \ast \ddot p_+(\ell, \cdot))(\ell) \hspace{1cm} ((\ell, m) \in \Z \times \N).
\end{align}
Observe that $q(\ell, m) = 0$, whenever $(\ell, m) \in \Z \setminus \Lbb \times \N_0$. In order to establish a non-recursive representation, for fixed $(\ell, z) \in \Z^2$, we define the (positive) convolution powers of the double sequence $( \ddot p_+(\ell, z) )_{(\ell, z) \in \Z^2}$ through $\ddot p_+^{\ast 0}(\ell, z) := \delta_{ \curbr{0} }(\curbr{z})$ and
\begin{align} \label{2025080101}
\ddot p_+^{\ast j}(\ell, z) := \suml_{z_1 = -\infty}^\infty \ddot p_+(\ell, z_1) \ddot p_+^{\ast (j-1)}(\ell-z_1, z - z_1) \hspace{1cm} (j \in \N).
\end{align}
The sum only consists of a finite number of summands, due to the assumptions on the sequence. Unlike the convolution power of a single-indexed sequence, however, the above generally does not commute. Instead, by induction, for each $(\ell, z) \in \Z^2$, one can easily show that
\begin{align} \label{2025080501}
\ddot p_+^{\ast j}(\ell, z) = \suml_{z_1 = -\infty}^\infty \ddot p_+(\ell-z_1, z-z_1) \ddot p_+^{\ast (j-1)}(\ell, z_1) \hspace{1cm} (j \in \N).
\end{align}
From (\ref{2025080101}), it is obvious that $\ddot p_+^{\ast j}(\ell, z) = 0$, for all $(j, \ell, z) \in \N_0 \times (\Z \setminus \Lbb) \times \Z$. Moreover, one inductively verifies that all terms equal zero, whose power $j$ exceeds the argument $z$, viz
\begin{align} \label{2025080102}
\ddot p_+^{\ast j}(\ell, z) = 0 \hspace{1cm} ((\ell, z) \in \Lbb \times \curbr{\hdots, j-2, j-1}).
\end{align}
It must be emphasized that the convolution powers $( p^{ \ast j}(\ell, z) )_{j \in \N_0}$ do not exhibit such a cancelling behaviour. Finally, the sum of the first $\floorfct{x}+1$ convolution powers is defined by
\begin{align} \label{2025072810}
\alpha\curbr{ \ddot p_+ }(\ell, z, x) := \suml_{j=0}^{ \lfloor x \rfloor } \ddot p_+^{\ast j}(\ell, z) \hspace{1cm} ((\ell, z, x) \in \Z^2 \times \R).
\end{align}
Clearly, $\alpha\curbr{ \ddot p_+ }(\ell, z, 0) = \delta_{ \curbr{0}}(\curbr{z})$, for all $(\ell, z) \in \Z^2$, as well as $\alpha\curbr{ \ddot p_+ }(\ell, z, x) = 0$, for $(\ell, z) \in (\Lbb \times -\N) \cup (\Z \setminus \Lbb \times \Z)$ and $x\in\R$. We are now ready to provide a definite representation for $q(\cdot, m)$.

\begin{lemma} \label{Lem20250805}
Under Assumption \ref{Ass2025081801}, for each $(\ell, m) \in \Z \times \N_0$, we have
\begin{align} \label{2025073001}
q(\ell,m) = ( \ddot r \ast \alpha\curbr{ \ddot p_+ }(\ell, \cdot, m) )(\ell).
\end{align}
\end{lemma}

\begin{proof}
We proceed by induction and remark that all series below are indeed finite, due to our assumptions. The case $m=0$ is clear. Supposing validity for $0, \hdots, m-1$, through (\ref{2025072812}), (\ref{2025073001}) and by induction hypothesis, we receive
\begin{align*}
q(\ell,m) = \ddot r(\ell) + \suml_{j=1}^m \suml_{z=-\infty}^\infty \ddot p_+(\ell, z) (\ddot r \ast \ddot p_+^{\ast (j-1)}(\ell-z, \cdot))(\ell-z).
\end{align*}
Substitution and additional rearrangements eventually yield
\begin{align*}
q(\ell,m) = \ddot r(\ell) + \suml_{j=1}^m \suml_{z_3=-\infty}^\infty \ddot r(\ell-z_3) \suml_{z=-\infty}^\infty \ddot p_+(\ell, z) \ddot p_+^{\ast (j-1)}(\ell-z, z_3-z).
\end{align*}
Upon accounting for the definition of convolution powers as in (\ref{2025080101}), we can write as well $q(\ell,m) = \ddot r(\ell) + \sum_{j=1}^m ( \ddot r \ast \ddot p_+^{\ast j}(\ell, \cdot) )(\ell)$, which was to show.
\end{proof}

Generally, Picard's iteration merely provides an approximation for the target, whose accuracy hopefully increases with the number of iterations. In the current situation, our approximation $q(\cdot, m)$ even coincides with the target $q|_\Lbb$, for all sufficiently large $m$, thanks to the decay of the convolution powers $( \ddot p_+^{\ast j}(\ell, z) )_{j \in \N_0}$, according to (\ref{2025080102}). Indeed, abbreviating
\begin{align} \label{2025051901}
\beta\curbr{ \ddot p_+}(\ell, z) := \alpha\curbr{ \ddot p_+ }(\ell, z, z) \hspace{1cm} ((\ell, z) \in \Z^2),
\end{align}
the following statement holds, which is one of the main results of the present paragraph.

\begin{theorem}[deconvolution I] \label{Theo2025073001}
Under Assumption \ref{Ass2025081801}, it holds that
\begin{align*}
q|_\Lbb(\ell) = ( \ddot r \ast \beta\curbr{\ddot p_+}(\ell, \cdot) )(\ell) \hspace{1cm} (\ell \in \Z).
\end{align*}
In particular, $q|_\Lbb(\ell) = q(\ell)$ for all $\ell \in \Z$, whenever $q(\ell) = 0$ for $\ell \in \Z \setminus \Lbb$.
\end{theorem}

In view of the last theorem, it is reasonable to speak of $( \beta\curbr{\ddot p_+}(\ell, z) )_{(\ell, z) \in \Z^2}$ as the inverse sequence to the convolution with $(p(\ell, z))_{(\ell, z) \in \Z^2}$.

\begin{proof}[Proof of Theorem \ref{Theo2025073001}]
Due to the fact that $q(\ell, m) = 0 = q|_\Lbb(\ell)$, for all $(\ell, z) \in \Z \setminus \Lbb \times \Z$, we may confine to $\ell \in \Lbb$. For each $\ell \in \Lbb$ and $x \geq 0$, by (\ref{2025080102}) and (\ref{2025072810}), we have $\alpha\curbr{ \ddot p_+ }(\ell, z, x) = \beta\curbr{ \ddot p_+ }(\ell, z)$, whenever $0 \leq z \leq \floorfct{x}$. Therefore, $q(\ell, m) = \sum_{j=0}^\ell ( \ddot r \ast \ddot p_+^{\ast j}(\ell, \cdot) )(\ell) = q(\ell, \ell)$, if $m \geq \ell$. At the same time, (\ref{2025072812}) implies that $q(\ell,\ell) = q(\ell, m+1) = \ddot r(\ell) + \sum_{z=0}^\ell q(\ell-z, \ell-z) \ddot p_+(\ell, z)$, for $m \geq \ell$, and hence $q(\ell,\ell) = q|_\Lbb(\ell)$, according to Lemma \ref{Lem2025072901}.
\end{proof}
 
Defining $\ddot p(\ell,z) := (p(\ell, 0))^{-1} p(\ell, z) \ONE_\Lbb(\ell)$, we briefly mention in passing that $( \ddot p^{\ast j}(\ell, z) )_{j \in \N_0}$ and $( \ddot p_+^{\ast j}(\ell, z) )_{j \in \N_0}$ are binomial transforms (see Appendix \ref{AppConvIdMeas}) of each other. More precisely, by induction, one can show that $\ddot p_+^{\ast j}(\ell, z) = \sum_{k=0}^j \binom{j}{k} (-1)^k \ddot p^{\ast k}(\ell, z)$, for $(j, \ell, z) \in \N_0 \times \Z^2$. Thereof, by additional use of the well-known binomial identity (26.3.7) in \cite{olver2010nist}, we get
\begin{align} \label{2025051601X}
\alpha\curbr{ \ddot p_+ }(\ell, z, x) = \suml_{k=0}^{ \floorfct{x} } \binom{ \floorfct{x} + 1 }{k+1} (-1)^k \ddot p^{\ast k}(\ell, z) \hspace{1cm} ((\ell, z, x) \in \Z^2 \times \R).
\end{align}
Finally, a special case deserves a separate emphasis, namely when the noise sequence merely depends on a single index, i.e., $p(\ell, z) = p(z, z) =: u(z)$, for all $(\ell, z) \in \Z^2$. In this event, if $u(0) \neq 0$, Theorem \ref{Theo2025073001} facilitates the recovery of the entire sequence $( q(\ell) )_{ \ell \in \Z}$, and the resulting representation involves ordinary convolution powers. To become more precise, for the sequence $( \ddot u_+(z) )_{z \in \Z}$ that is defined by
\begin{align} \label{2025091102}
\ddot u_+(z) := \delta_{ \curbr{0} }(\curbr{z}) - \frac{ u(z) }{ u(0) } \hspace{1cm} (z \in \Z),
\end{align}
we denote the sum of the first $z+1$ convolution powers by
\begin{align} \label{2025091103}
\gamma\curbr{ \ddot u_+ }(z) := \suml_{j = 0}^z \ddot u_+^{\ast j}(z) \hspace{1cm} (z \in \Z).
\end{align}
Notice that the right hand side can be expanded via the binomial convolution theorem, Lemma \ref{LemDekmkompakt02}, to obtain a representation similar to (\ref{2025051601X}). In any case, the following result is straightforward.

\begin{corollary}[deconvolution II] \label{Lem2025091101}
If Assumption \ref{Ass2025081801} holds, with $p(\ell, z) = p(z, z) =: u(z)$, for all $(\ell, z) \in \Z^2$, and $u(0) \neq 0$, then
\begin{align*}
q(\ell) = (\ddot r \ast \gamma\curbr{ \ddot u_+ })(\ell) \hspace{1cm} (\ell \in \Z).
\end{align*}
In particular, $\ddot u_+^{\ast j}(z) = 0$, for all $(j, z) \in \N_0 \times \Z$ with $z \leq j-1$.
\end{corollary}

\begin{proof}
The assumptions imply that $\Lbb = \N_0$ and $r = u \ast q$, as well as that $u(z) = 0$, for $z \in -\N$. Now, since the asserted identity is obvious for $\ell \in -\N$, without loss of generality, we suppose that $\ell \in \N_0$. Then, $\ddot p_+(\ell, z) = \ddot u_+(z)$, for all $z\in \Z$. We thus conclude from (\ref{2025080101}) that even $\ddot p_+^{\ast j}(\ell, z) = \ddot u_+^{\ast j}(z)$, for all $(j, z) \in \N_0 \times \Z$, and thereby $\beta\curbr{ \ddot p_+ }(\ell, z) = \gamma\curbr{ \ddot u_+ }(z)$. The corollary hence directly follows from the properties of $( \ddot p_+^{\ast j}(\ell, z) )_{j \in \N_0}$ and from Theorem \ref{Theo2025073001}.
\end{proof}

We proceed with two applications of Theorem \ref{Theo2025073001} to various settings of errors in variables. In both of them, the mass of $F_X$ is concentrated on a left-bounded monotonic set. Our examples will not only provide an exact representation for $F_X$, but even facilitate the unbiased estimation by means of an i.i.d. sample of the blurred variable $Y$.

\begin{corollary}[left-bounded monotonic $\Tbb_X$ and left-bounded countable $\Tbb_\varepsilon$] \label{SatzdiskrDekWkeitsfkt}
Suppose that $\Tbb_X \subseteq \curbr{ \xi_\ell }_{\ell \in \N_0}$, with $\xi_{\ell-1} < \xi_\ell$, for all $\ell \in \N$, and $F_X\curbr{\xi_0} > 0$. Additionally, let $\Tbb_\varepsilon \subset \R$ be countable and assume the existence of $z_0\in\Tbb_\varepsilon$ with $F_\varepsilon(z_0-) = 0$. Define the sequences $( p_X(\ell) )_{\ell \in \Z}$, $( \ddot p_{\varepsilon, +}(\ell, z) )_{ (\ell, z) \in \Z^2}$ and $(\ddot p_Y(\ell) )_{\ell \in \Z}$ by $p_X(\ell) := F_X\curbr{\xi_\ell} \ONE_{\N_0}(\ell)$,
\begin{align*}
\ddot p_{\varepsilon, +}(\ell, z) :=&\, \rb{ \delta_{\curbr{0}}(\curbr{z}) - \frac{ F_\varepsilon\curbr{z_0 + \xi_\ell-\xi_{\ell-z}} }{ F_\varepsilon\curbr{z_0} } \ONE_{\curbr{0, \hdots, \ell}}(z) } \ONE_{\N_0}(\ell), \\
\ddot p_Y(\ell) :=&\, \frac{ F_Y\curbr{z_0 + \xi_\ell} }{ F_\varepsilon\curbr{z_0} } \ONE_{\N_0}(\ell).
\end{align*}
Then,
\begin{align*}
p_X(\ell) = (\ddot p_Y \ast \beta\curbr{\ddot p_{\varepsilon, +}}(\ell, \cdot))(\ell) \hspace{1cm} (\ell \in \Z).
\end{align*}
Specifically if there exists $s> 0$ with $\xi_\ell = \xi_0 + s\ell$, for all $\ell \in \N_0$, defining $\ddot u_{\varepsilon, +}(z) := \ddot p_{\varepsilon, +}(z, z)$, it holds that
\begin{align*}
F_X(\xi) = (\Theta\curbr{\ddot p_Y} \ast \Theta\curbr{ \gamma\curbr{\ddot u_{\varepsilon, +} } })\rb{\frac{\xi-\xi_0}{s}} \hspace{1cm} (\xi \in \R).
\end{align*}
\end{corollary}

Notice that the span $s > 0$ in the equidistant case of Corollary \ref{SatzdiskrDekWkeitsfkt} need not be unique. Moreover, monotonicity is only required for the support of the target variable, whereas the support of the errors may even be dense. Figure \ref{ExPlotsDecUnilatBothDisc} displays a few examples, where an unbiased estimator for $F_X$ was constructed from this result. Besides errors with a Poisson distribution, we also included discretely uniformly distributed errors. The sharp contrast between these two distributions becomes obvious in the Fourier domain. While the c.f. of any Poisson distribution is never zero, the c.f. of any discrete uniform distribution has infinitely many zeros.

\begin{proof}[Proof of Corollary \ref{SatzdiskrDekWkeitsfkt}]
In the described situation, $\Tbb_Y = \curbr{x+z : x \in \Tbb_X, z \in \Tbb_\varepsilon}$ and $z_0+\xi_0$ is the left extremity of $\Tbb_Y$. Moreover, the convolution equation (\ref{MassFaltung}) implies that $F_Y\curbr{z_0+\xi_\ell} = \sum_{i=0}^\ell F_X\curbr{\xi_i}F_\varepsilon\curbr{z_0+\xi_\ell-\xi_i}$, for $\ell \in \N_0$. Accordingly, $p_Y(\ell) = (p_X \ast p_\varepsilon)(\ell)$, for $\ell \in \Z$, in terms of $p_Y(\ell) := F_Y\curbr{z_0+\xi_\ell}\ONE_{ \N_0 }(\ell)$, $p_\varepsilon(\ell, z) := F_\varepsilon\curbr{z_0+\xi_\ell-\xi_{\ell-z}}\ONE_{ \curbr{0, \hdots, \ell} }(z) \ONE_{ \N_0 }(\ell)$ and with $p_X$ as in the theorem. Since these sequences satisfy Assumption \ref{Ass2025081801}, with $\Lbb= \N_0$, the first part directly follows from Theorem \ref{Theo2025073001}. Lastly, if $\xi_\ell = \xi_0 + s\ell$, it is obvious that $F_Y\curbr{z_0+\xi_0+s\ell} = \sum_{i=0}^\ell F_X\curbr{\xi_0+si}F_\varepsilon\curbr{z_0+s(\ell-i)}$ and hence Corollary \ref{Lem2025091101} applies. Also, then $F_X(\xi) = \sum_{ \ell = 0 }^\infty \ONE_{ \curbr{ \xi_0+s\ell \leq \xi } } F_X\curbr{\xi_0+s\ell} = \Theta\curbr{p_X}(s^{-1}(\xi-\xi_0))$, for all $\xi \in \R$.
\end{proof}

\begin{figure}[h]
\centering
\resizebox*{\textwidth}{!}{\includegraphics{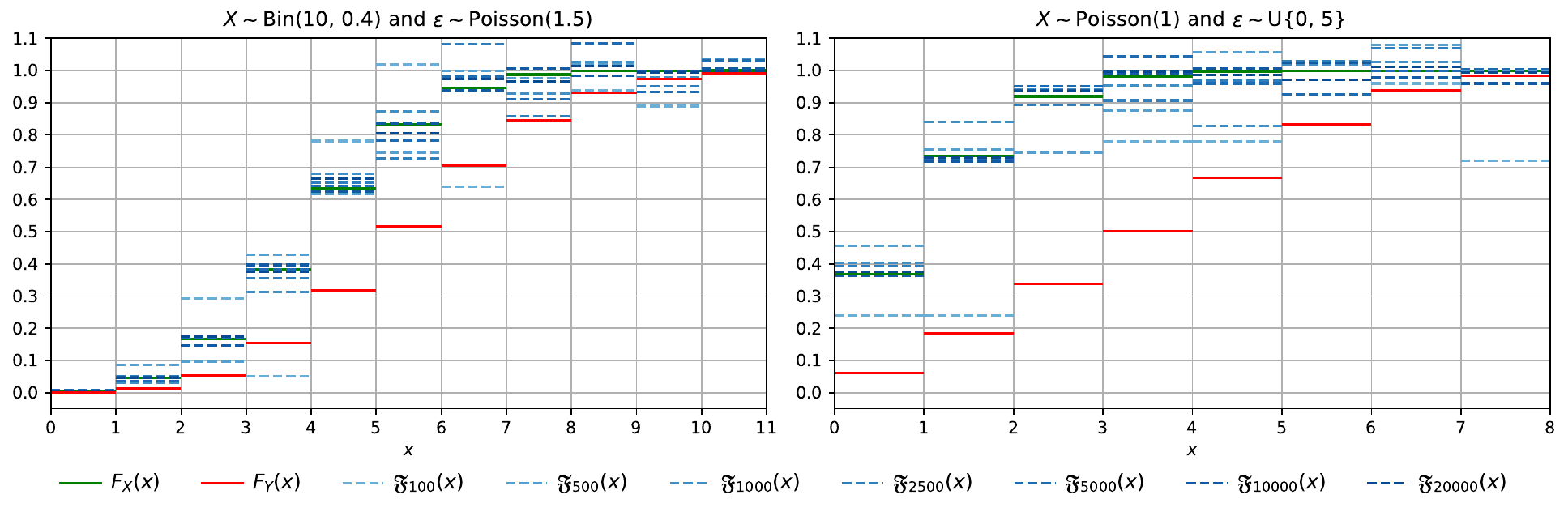}}
\caption{Plots for target d.f. $F_X$, blurred d.f. $F_Y$ and plug-in estimator of inverse, in various purely discrete setups that match Corollary \ref{SatzdiskrDekWkeitsfkt}. In all these cases, $s=1$, $\xi_0=z_0=0$ and $\Theta\curbr{\ddot p_Y} = (F_\varepsilon\curbr{0})^{-1} F_Y$, so that the estimator is given by $\mathfrak{F}_n := (n F_\varepsilon\curbr{0})^{-1} \sum_{i=1}^n \Theta\curbr{ \gamma\curbr{\ddot u_{\varepsilon, +} } }(\cdot - Y_i)$, for an i.i.d. sample $Y_1, \hdots, Y_n$ of size $n \in \N$.}
\label{ExPlotsDecUnilatBothDisc}
\end{figure}

The situation from the last corollary is also included in the next, thereby showing that the inverse need not be unique. Actually, in the next case, there exist infinitely many inverse sequences. It is dedicated to discrete $F_X$ and arbitrary $F_\varepsilon$, both right-lateral.

\begin{corollary}[left-bounded monotonic $\Tbb_X$ and arbitrary left-bounded $\Tbb_\varepsilon$] \label{BspXdiskrZstet}
Suppose that $\Tbb_X \subseteq \curbr{ \xi_\ell }_{\ell \in \N_0}$, with $\xi_{\ell-1} < \xi_\ell$, for all $\ell \in \N$, and $F_X\curbr{\xi_0} > 0$, as well as the existence of $z_0 \in \R$, with $F_\varepsilon(z_0) = 0$. Choose $\curbr{ \zeta_\ell }_{\ell \in \N_0} \subset \R$, such that $\xi_\ell < \zeta_\ell \leq \xi_{\ell+1}$ and $F_\varepsilon(z_0 + \zeta_\ell-\xi_\ell) > 0$, for each $\ell \in \N_0$. Moreover, define the sequences $( p_X(\ell) )_{\ell \in \Z}$, $( \ddot P_{\varepsilon, +}(\ell, z) )_{ (\ell, z) \in \Z^2}$ and $( \ddot P_Y(\ell) )_{\ell \in \Z}$ by $p_X(\ell) := F_X\curbr{\xi_\ell} \ONE_{ \N_0 }(\ell)$,
\begin{align*}
\ddot P_{\varepsilon, +}(\ell, z) :=&\, \rb{ \delta_{\curbr{0}}(\curbr{z}) - \frac{ F_\varepsilon(z_0 + \zeta_\ell-\xi_{\ell-z}) }{ F_\varepsilon(z_0 + \zeta_\ell-\xi_\ell) } \ONE_{\curbr{0, \hdots, \ell}}(z) } \ONE_{\N_0}(\ell), \\
\ddot P_Y(\ell) :=&\, \frac{ F_Y(z_0 + \zeta_\ell) }{ F_\varepsilon(z_0 + \zeta_\ell-\xi_\ell) } \ONE_{ \N_0 }(\ell).
\end{align*}
Then,
\begin{align*}
p_X(\ell) = (\ddot P_Y \ast \beta\curbr{\ddot P_{\varepsilon, +}}(\ell, \cdot))(\ell) \hspace{1cm} (\ell \in \Z).
\end{align*}
Particularly if there exist $s > 0$ and $0 < \sigma \leq s$, such that $\xi_\ell = \xi_0 + s\ell$, for all $\ell \in \N_0$, and $F_\varepsilon(z_0+\sigma) > 0$, defining $\zeta_\ell := \xi_0 + \sigma + s\ell$ and $\ddot U_{\varepsilon, +}(z) := \ddot P_{\varepsilon, +}(z, z)$, we have
\begin{align*}
F_X(\xi) = (\Theta\curbr{\ddot P_Y} \ast \Theta\curbr{\gamma\curbr{\ddot U_{\varepsilon, +}}})\rb{\frac{\xi-\xi_0}{s}} \hspace{1cm} (\xi \in \R).
\end{align*}
\end{corollary}

\begin{figure}[h]
\centering
\resizebox*{\textwidth}{!}{\includegraphics{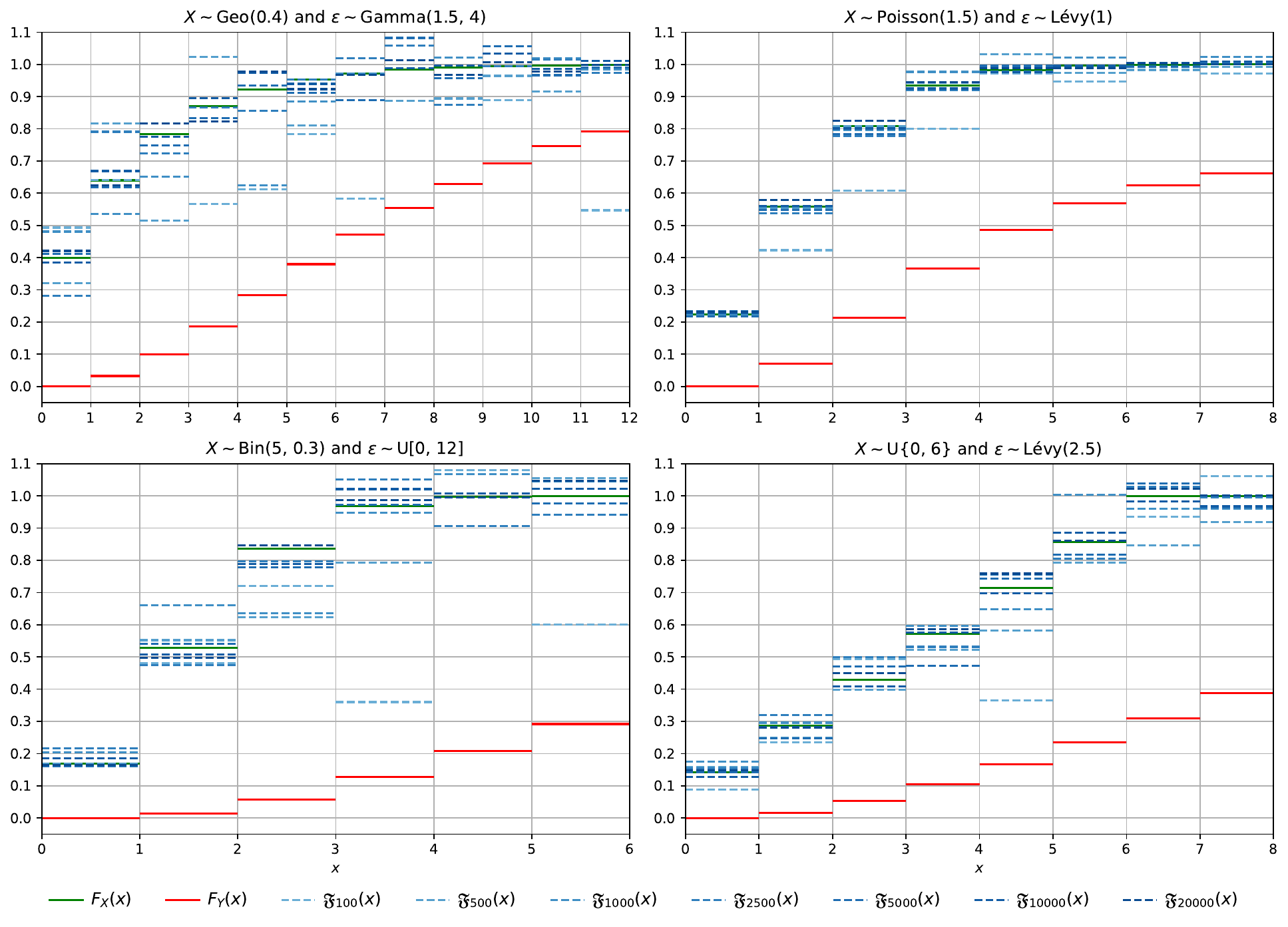}}
\caption{The above plots consider situations that match Corollary \ref{BspXdiskrZstet}, for distributions with $s=1$, $\xi_0 = z_0 = 0$ and $\zeta_\ell := \xi_{\ell+1}$. Then, defining $\ddot P_{Y, n}(\ell) := ( F_\varepsilon(1) )^{-1} F_Y(1+\ell, n)$, an unbiased deconvolution estimator for $F_X$ is given by $\mathfrak{F}_n := \Theta\curbr{\ddot P_{Y, n}} \ast \Theta\curbr{\gamma\curbr{\ddot U_{\varepsilon, +}}}$. To decrease computational effort, it is helpful to observe that $\mathfrak{F}_n(\xi) = (nF_\varepsilon(1))^{-1} \sum_{i=1}^n \sum_{j=0}^{ \floorfct{\xi} } \Theta\curbr{ \gamma\curbr{\ddot U_{\varepsilon, +}} }(1+j-Y_i)$, for $\xi \in \R$ and an i.i.d. sample $Y_1, \hdots, Y_n$.}
\label{ExPlotsDecUnilatContDisc}
\end{figure}

In the situation of Corollary \ref{BspXdiskrZstet}, $F_Y$ need not have atoms. Therefore, starting point for the proof will be the convolution equation for d.fs., rather than for probability functions. A few selected examples for applications to the estimation of $F_X$ are drawn in Figure \ref{ExPlotsDecUnilatContDisc}.

\begin{proof}[Proof of Corollary \ref{BspXdiskrZstet}]
By assumption, $F_Y(z_0+\zeta_\ell) = \sum_{i=0}^\ell F_X\curbr{\xi_i} F_\varepsilon(z_0+\zeta_\ell-\xi_i)$, for each $\ell \in \N_0$. Hence, defining $P_Y(\ell) := F_Y(z_0+\zeta_\ell) \ONE_{\N_0}(\ell)$, $P_\varepsilon(\ell, z) := F_\varepsilon(z_0+\zeta_\ell-\xi_{\ell-z}) \ONE_{ \curbr{0, \hdots, \ell} }(z) \ONE_{ \N_0 }(\ell)$ and with $p_X$ as in the theorem, the last equation implies that $P_Y(\ell) = (p_X \ast P_\varepsilon(\ell, \cdot))(\ell)$, for all $\ell \in \Z$. It shows that Assumption \ref{Ass2025081801} holds, again with $\Lbb = \N_0$, and thus Theorem \ref{Theo2025073001} applies, of which the first part is an immediate consequence. The second part eventually follows from Corollary \ref{Lem2025091101}, since then $F_Y(z_0+\xi_0+\sigma+s\ell) = \sum_{i=0}^\ell F_X\curbr{\xi_0+si} F_\varepsilon(z_0+\sigma+s(\ell-i))$.
\end{proof}

Another special case, that may almost be overlooked, owing to its simplicity, is that of a degenerate target sequence. Indeed, Corollary \ref{Lem2025091101} directly provides a representation for the identity of discrete convolution, in terms of any right-lateral sequence.

\begin{lemma} \label{LemRepIndicator}
For any sequence $( u(z) )_{z \in \Z} \subset \C$, such that $u(z) = 0$, for $z \in -\N$, and $u(0) \neq 0$, with $( \ddot u_+(z) )_{z \in \Z}$ as in (\ref{2025091102}), we have
\begin{align*}
\delta_{ \curbr{0} }(\curbr{\ell}) = ( u(0) )^{-1} ( u \ast \gamma\curbr{\ddot u_+} )(\ell) \hspace{1cm} (\ell \in \Z).
\end{align*}
\end{lemma}

\begin{proof}
Consider the identity $u = \delta_{\curbr{0}}\ast u$. Then, Assumption \ref{Ass2025081801} holds, with $p(\ell, z) = u(z)$, for all $(\ell, z) \in \Z^2$, so that we may immediately apply Corollary \ref{Lem2025091101}. The claimed identity is now obvious from the fact that $\ddot r(\ell)= (u(0))^{-1}u(\ell)$, for all $\ell \in \Z$.
\end{proof}

In some cases, the sum representation for the inverse sequence $( \gamma\curbr{\ddot u_+}(z))_{z \in \Z}$ may simplify. We conclude this paragraph with a few noteworthy examples.

\begin{example}\label{Ex20251016}
In all our subsequent examples, we apply Lemma \ref{LemDiscConvPwrs} to the measure with atoms at $(\ddot u_+(z))_{z \in \Z}$. We also exploit the fact that the number of compositions of $\ell \in \N$ into $j \in \N$ parts equals $\binom{\ell-1}{j-1}$ \cite[see][Example 1.6]{Flajolet_Sedgewick_2009}.
\begin{enumerate}
\item Bernoulli sequence: If $u(z) \in \C \setminus \curbr{0}$, for $z \in \curbr{0, 1}$, and $u(z) = 0$ else, we readily infer that $\ddot u_+^{\ast j}(z) = (u(0))^{-j} ( -u(1) )^j \delta_{ \curbr{j} }(\curbr{z})$, for each $j \in \N_0$, which in turn implies that
\begin{align} \label{2025101401}
\gamma\curbr{ \ddot u_+ }(z) = \rrb{ - \frac{u(1)}{u(0)} }^z \ONE_{\N_0}(z).
\end{align}
\item Geometric sequence: For fixed $u \in \C \setminus \curbr{0}$, let $u(z) :=  u(1-u)^z \ONE_{\N_0}(z)$. Then, $\ddot u_+^{\ast j}(z) = (-1)^j \binom{z-1}{j-1} (1-u)^z \ONE_\N(z) \ONE_{ \curbr{1, \hdots, z} }(j)$, for $j \in \N$. From the binomial theorem, we therefore deduce that
\begin{align} \label{2025101602}
\gamma\curbr{ \ddot u_+ }(z) = \delta_{ \curbr{0} }(\curbr{z}) - (1-u)\delta_{ \curbr{1} }(\curbr{z}).
\end{align}
\item Poisson sequence: For fixed $\lambda > 0$, define $u(z) := e^{-\lambda} (z!)^{-1} \lambda^z \ONE_{\N_0}(z)$. Notice that $\sum_{k=0}^j \binom{j}{k} (-1)^k k^z = 0$, for all $j \in \N$ and $z \in \curbr{0, \hdots, j-1}$. With this, for $j \in \N$, one can show that
\begin{align*}
\ddot u_+^{\ast j}(z) = \frac{\lambda^z}{z!} \ONE_\N(z) \ONE_{ \curbr{1, \hdots, z} }(j) \suml_{k=1}^j \binom{j}{k} (-1)^k k^z,
\end{align*}
from which it follows that
\begin{align} \label{2025102501}
\gamma\curbr{ \ddot u_+ }(z) = \frac{ (-\lambda)^z }{ z! } \ONE_{\N_0}(z).
\end{align}
\item Uniform sequence: For fixed $K \in \N$ and $u \in \C \setminus \curbr{0}$, define $u(z) := u \ONE_{ \curbr{0, \hdots, K} }(z)$. Then, $\ddot u_+^{\ast j}(z) = (-1)^j \binom{z-1}{j-1} \ONE_{ \curbr{j, j+1, \hdots, jK } }(z)$, for $j \in \N$. Thus,
\begin{align*}
\gamma\curbr{ \ddot u_+ }(z) &= \delta_{ \curbr{0} }(\curbr{z}) + \ONE_{ \N }(z) \suml_{ j = \lceil \frac{z}{K} \rceil }^z \binom{z-1}{j-1} (-1)^j.
\end{align*}
For $1 \leq z \leq K$, the sum simplifies through the binomial theorem. For $z \geq K+1$, by application of Pascal's rule, it cancels to a single addend. Altogether, we find that
\begin{align} \label{2025101701}
\gamma\curbr{ \ddot u_+ }(z) = \delta_{ \curbr{0} }(\curbr{z}) - \delta_{ \curbr{1} }(\curbr{z}) + (-1)^{ \lceil \frac{z}{K} \rceil } \binom{ z - 2 }{ \lceil \frac{z}{K} \rceil - 2 } \ONE_{ \curbr{K+1, K+2, \hdots } }(z).
\end{align}
For $K=1$, the result matches (\ref{2025101401}), and the sequence $\gamma\curbr{\ddot u_+}(z)$ does not converge, as $z \rightarrow \infty$, but it remains bounded. A complete asymptotic discussion in the case $K\geq 2$ requires use of Stirling's approximation. We briefly verify divergence in a simpler way. For $N \in \N$, by elementary manipulations, it is easy to see that
\begin{align*}
\binom{KN-2}{N-2} = (N-1) \prodl_{j=2}^{N-1} \rb{\frac{KN}{j}-1},
\end{align*}
which obviously grows to infinity, as $N \rightarrow \infty$. Therefore, also $|\gamma\curbr{ \ddot u_+ }(KN)| \rightarrow \infty$, as $N \rightarrow \infty$.
\end{enumerate}
\end{example}

Suppose for a moment that all sequences of Example \ref{Ex20251016} are associated with probability distributions, i.e., $( u(z) )_{z \in \Z} \subset [0, 1]$ with $\sum_{z\in \Z} u(z) = 1$. Then, in case one, the measure with mass at $( \gamma\curbr{\ddot u_+}(z) )_{z \in \Z}$ is of finite total variation on $\R$, if and only if $u(0) > \frac{1}{2}$. In the second and third case, the measure is even always of finite total variation on $\R$. In particular, the associated Fourier-Stieltjes transforms exist in both cases, and one can easily show that they coincide with the reciprocal c.f. of the original sequence $(u(z))_{z \in \Z}$, up to the constant factor $u(0)$. Lastly, in the fourth case, the measure induced by the inverse sequence is of infinite total variation on $\R$. At the same time, it is well-known that the reciprocal c.f. of any non-degenerate discrete uniform distribution is unbounded. We conclude that an inverse operator may exist in the domain of d.fs., although this is not indicated in the Fourier domain.

\subsection{A noise with an equidistant support}

The actually interesting point about the degenerate case is that it ultimately gives rise to a deconvolution theorem that is not restricted to sequences. First of all, from Lemma \ref{LemRepIndicator}, for any $Q \in L^\infty(\R)$ and $T \in \N_0$, we infer that
\begin{align} \label{2025100801}
Q(\xi) = (u(0))^{-1} \suml_{z_2 = 0}^T Q(\xi-z_2) \suml_{z=0}^{z_2} u(z_2-z) \gamma\curbr{ \ddot u_+ }(z) \hspace{1cm} (\xi \in \R).
\end{align}
Generally, since the behaviour of $\gamma\curbr{ \ddot u_+ }(z)$, as $z \rightarrow \infty$, is non-trivial, the behaviour of the sequence of double sums, as $T \rightarrow \infty$, is more or less arbitrary. In fact, Example \ref{Ex20251016} suggests that the divergence of this sequence is not unusual. Accordingly, the asymptotic behaviour of the above double sum substantially depends on the function $Q$. Upon interchanging the summation order, we arrive at
\begin{align} \label{2025100802}
Q(\xi) = (u(0))^{-1} \suml_{z = 0}^T \gamma\curbr{ \ddot u_+ }(z) \suml_{z_3=0}^{T-z} Q(\xi-z-z_3) u(z_3) \hspace{1cm} (\xi \in \R).
\end{align}
The obtained representation immediately facilitates a characterization of convergence, of which the following statement is a direct consequence.

\begin{theorem}[deconvolution III] \label{TheoDecII}
For an arbitrary $Q \in L^\infty(\R)$ and a sequence $( u(z) )_{z \in \Z} \subset \C$, with $u(z) = 0$, for all $z \in -\N$, and $u(0) \neq 0$, denote $R := Q \ast \Theta\curbr{u}$ and let $( \ddot u_+(z) )_{z \in \Z}$ be defined as in (\ref{2025091102}). Then,
\begin{align*}
Q(\xi) = (u(0))^{-1} (R \ast \Theta\curbr{ \gamma\curbr{\ddot u_+} })(\xi) \hspace{1cm} (\xi \in \R),
\end{align*}
whenever one of the following conditions is fulfilled:
\begin{enumerate}
\item It exists $\xi_0 \in \R$ with $Q(\xi) = 0$ for all $\xi < \xi_0$.
\item $Q$ is non-decreasing on $\R$, $( u(z) )_{z \in \Z} \in \ell^1(\Z)$ and $( \gamma\curbr{ \ddot u_+ }(z) Q(\xi-z) )_{z \in \Z} \in \ell^1(\Z)$, for each $\xi \in \R$.
\end{enumerate}
\end{theorem}

\begin{proof}
In the first case, $(R \ast \Theta\curbr{ \gamma\curbr{\ddot u_+} })(\xi) = \sum_{z=0}^{\floorfct{\xi-\xi_0}} \gamma\curbr{\ddot u_+}(z) \sum_{z_3=0}^{ \floorfct{\xi-\xi_0} - z } Q(\xi-z-z_3) u(z_3)$, which matches the sum on the right hand side of (\ref{2025100802}), for all $T > \floorfct{\xi-\xi_0}$. Hence, the asserted identity is obvious. In the second case, appealing to the monotonicity of $Q$, we get $\sum_{z_3=0}^\infty | Q(\xi-z-z_3)u(z_3) | \leq |Q(\xi-z)| \sum_{z_3=0}^\infty | u(z_3) | < \infty$. Thereby, we deduce that
\begin{align*}
\sum_{z = 0}^\infty |\gamma\curbr{ \ddot u_+ }(z)| \sum_{z_3=0}^\infty |Q(\xi-z-z_3) u(z_3) | \leq \sum_{z_3=0}^\infty |u(z_3) | \sum_{z = 0}^\infty |\gamma\curbr{ \ddot u_+ }(z)| |Q(\xi-z)| < \infty.
\end{align*}
Altogether, we conclude absolute and with respect to $T > 0$ uniform convergence of the double sum (\ref{2025100802}). Hence, considering the limit as $T \rightarrow \infty$, we may interchange the order of limit and summation. In this, $R(\xi-z) = \sum_{z_3=0}^\infty Q(\xi-z-z_3) u(z_3)$, which confirms the desired identity.
\end{proof}

In practice, a verification of the conditions of Theorem \ref{TheoDecII} can be quite hard. Before we underline the necessity of these conditions, we apply the result to a probabilistic setup, with almost no assumptions on the target distribution.

\begin{corollary}[left-bounded equidistant $\Tbb_\varepsilon$] \label{BspXbelZdiskr}
Suppose that $\Tbb_\varepsilon \subseteq \curbr{z_0+tz}_{z \in \N_0}$, for $z_0 \in \R$ and $t > 0$, with $F_\varepsilon\curbr{z_0} > 0$. Define $P_Y(\xi) := F_Y(z_0+t\xi)$ and
\begin{align*}
\ddot u_{\varepsilon, +}(\ell) := \delta_{ \curbr{0} }(\curbr{\ell}) - \frac{F_\varepsilon\curbr{z_0+t\ell}}{F_\varepsilon\curbr{z_0}}.
\end{align*}
Additionally assume that $( \gamma\curbr{\ddot u_{\varepsilon, +} }(z) F_X(t(\xi-z)) )_{z \in \Z} \in \ell^1(\Z)$, for each $\xi \in \R$. Then,
\begin{align*}
F_X(\xi) = ( F_\varepsilon\curbr{z_0} )^{-1}(P_Y \ast \Theta\curbr{ \gamma\curbr{\ddot u_{\varepsilon, +}} })( t^{-1} \xi ) \hspace{1cm} (\xi \in \R).
\end{align*}
\end{corollary}

Notice that the condition on $F_X$ is always true if the support $\Tbb_X$ is bounded to the left. Plots for the estimation of such d.fs. by means of Corollary \ref{BspXbelZdiskr} can be found in Figure \ref{ExPlotsDecUnilatDiscCont}.

\begin{figure}[h]
\centering
\resizebox*{\textwidth}{!}{\includegraphics{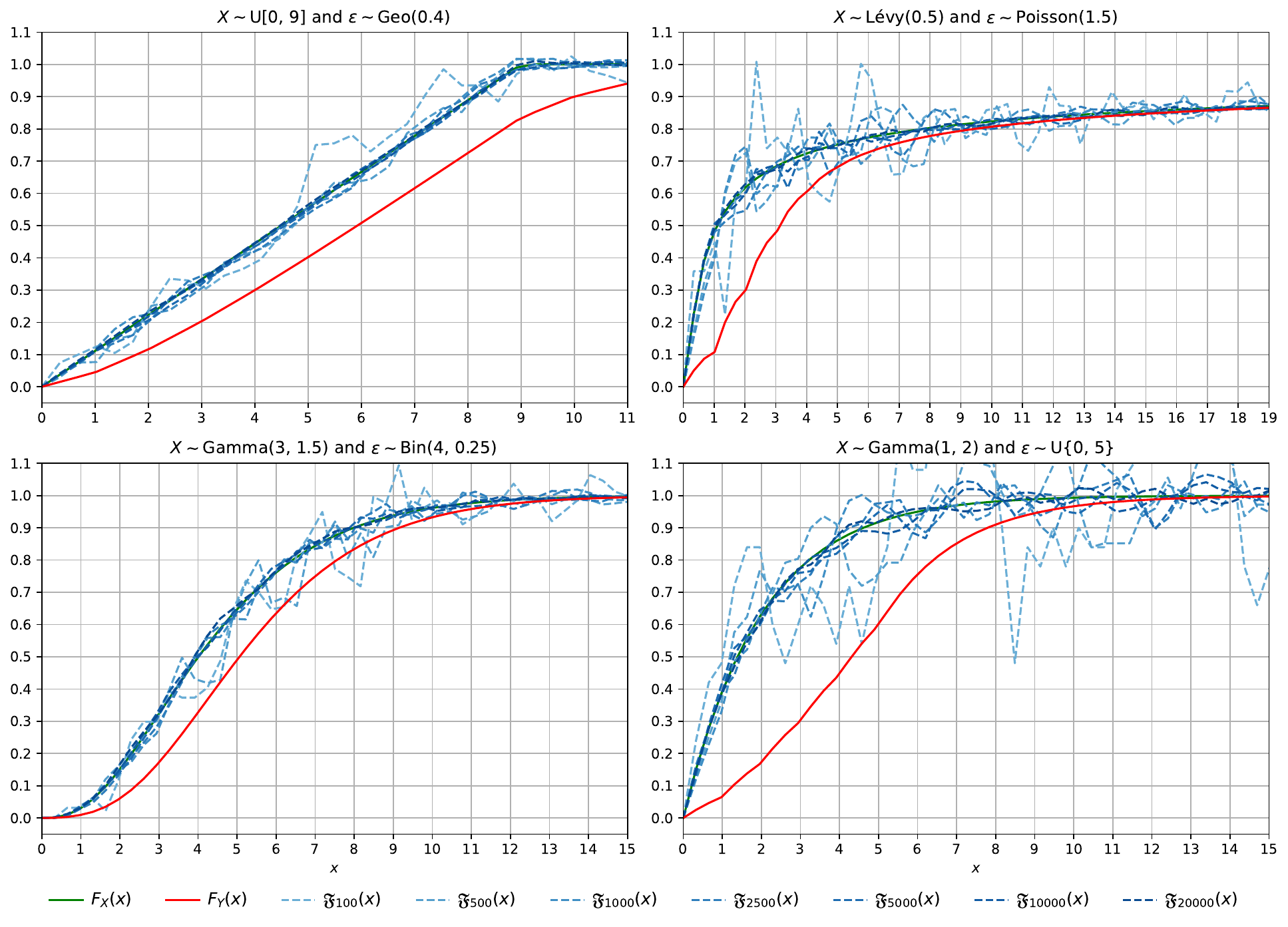}}
\caption{Plots for the deconvolution of a right-lateral continuous target d.f. $F_X$ from a blurred d.f. $F_Y$. The plug-in estimator $\mathfrak{F}_n$ for $F_X$ was constructed from Corollary \ref{BspXbelZdiskr}, based on a $Y$-sample of size $n \in \N$. In all cases, $(z_0, t) = (0, 1)$, so that $P_Y = F_Y$.}
\label{ExPlotsDecUnilatDiscCont}
\end{figure}

\begin{proof}[Proof of Corollary \ref{BspXbelZdiskr}]
In the described setting, $F_Y(z_0 + t\xi) = \sum_{z=0}^\infty F_\varepsilon\curbr{z_0+tz} F_X(t(\xi-z))$, for all $\xi \in \R$. Put differently, in terms of $R(\xi) := F_Y(z_0+t\xi)$, $Q(\xi) := F_X(t\xi)$ and $u_\varepsilon(z) := F_\varepsilon\curbr{z_0 + tz}$, we have $R = Q \ast \Theta\curbr{u_\varepsilon}$. In this, $Q$ is non-decreasing, $u_\varepsilon(z) = 0$, for $z \in -\N$, $u_\varepsilon(0) \neq 0$ and $(u_\varepsilon(z))_{z \in \Z} \in \ell^1(\Z)$. The asserted identity therefore directly follows from Theorem \ref{TheoDecII}.
\end{proof}

Figure \ref{ExPlotsDecBilatDiscCont} additionally displays plots for the estimated d.f. of a bilateral distribution. These were crafted straightforwardly, without a preliminary check of the respective condition of Theorem \ref{TheoDecII}. The results look promising and suggest that the theorem is indeed applicable. Yet, our next example warns us about a careless use of this theorem.

\begin{example}
Suppose that $X \sim \operatorname{Laplace}(0, \sigma)$ and $\varepsilon \sim \operatorname{Ber}(p)$, for $\sigma > 0$ and $0 < p < 1$. In these circumstances, $F_X(\xi) = \frac{1}{2} \exp\curbr{\frac{\xi}{\sigma}}$, for $\xi \leq 0$, and $F_\varepsilon\curbr{z} = p\delta_{ \curbr{1} }(\curbr{z}) + (1-p)\delta_{ \curbr{0} }(\curbr{z})$. Thus, $F_Y(\xi) = \kappa \exp\curbr{\frac{\xi}{\sigma}}$, for $\kappa := p\exp\curbr{-\frac{1}{\sigma}}+1-p$ and $\xi \leq 0$. Defining $u_\varepsilon(z) := F_\varepsilon\curbr{z}$, we will now show that $(F_Y \ast \Theta\curbr{ \gamma\curbr{\ddot u_{\varepsilon, +}} })(\xi)$ is unspecified, for all $\xi \leq 0$. First, for fixed $T \in \N_0$, by means of (\ref{2025101401}), we get
\begin{align*}
\suml_{z = 0}^T \gamma\curbr{\ddot u_{\varepsilon, +}}(z) F_Y(\xi-z) = \kappa e^{ \frac{\xi}{\sigma} } \suml_{z=0}^T (-1)^z e^{ z \curbr{ \log\rb{\frac{p}{1-p}} - \frac{1}{\sigma} } }.
\end{align*}
Letting $b := \log(p)-\log(1-p)-\sigma^{-1}$, for brevity, the geometric sum formula yields
\begin{align*}
\suml_{z = 0}^T \gamma\curbr{\ddot u_{\varepsilon, +}}(z) F_Y(\xi-z) = \kappa e^{ \frac{\xi}{\sigma} } \frac{ 1-(-e^b)^{T+1} }{1+e^b}.
\end{align*}
Now, the limit, as $T \rightarrow \infty$, of the left hand side is just equal to $(F_Y \ast \Theta\curbr{ \gamma\curbr{\ddot u_{\varepsilon, +}} })(\xi)$. However, the right hand side suggests that this limit may not exist. More precisely, it exists if and only if $b<0$. Conversely, if $b \geq 0$, i.e., $p > \frac{1}{2}$ and $\sigma \geq (\log(p)-\log(1-p))^{-1}$, the limit as $T \rightarrow \infty$ is undefined. In fact, if $b > 0$, the sum is even unbounded with respect to $T$.
\end{example}

To summarize the current paragraph, in some circumstances, the deconvolution of the unknown d.f. $F_X$ is possible and hence even the unbiasd estimation. Yet, the derivation of the above deconvolution formulae essentially exploits the assumed structure of the involved distributions. Aiming for a broader applicability, in the next paragraph, we will shed a new light on the convolution equation in an arbitrary framework, from which eventually a generalization of the above results will be obtained.

\begin{figure}[h]
\centering
\resizebox*{\textwidth}{!}{\includegraphics{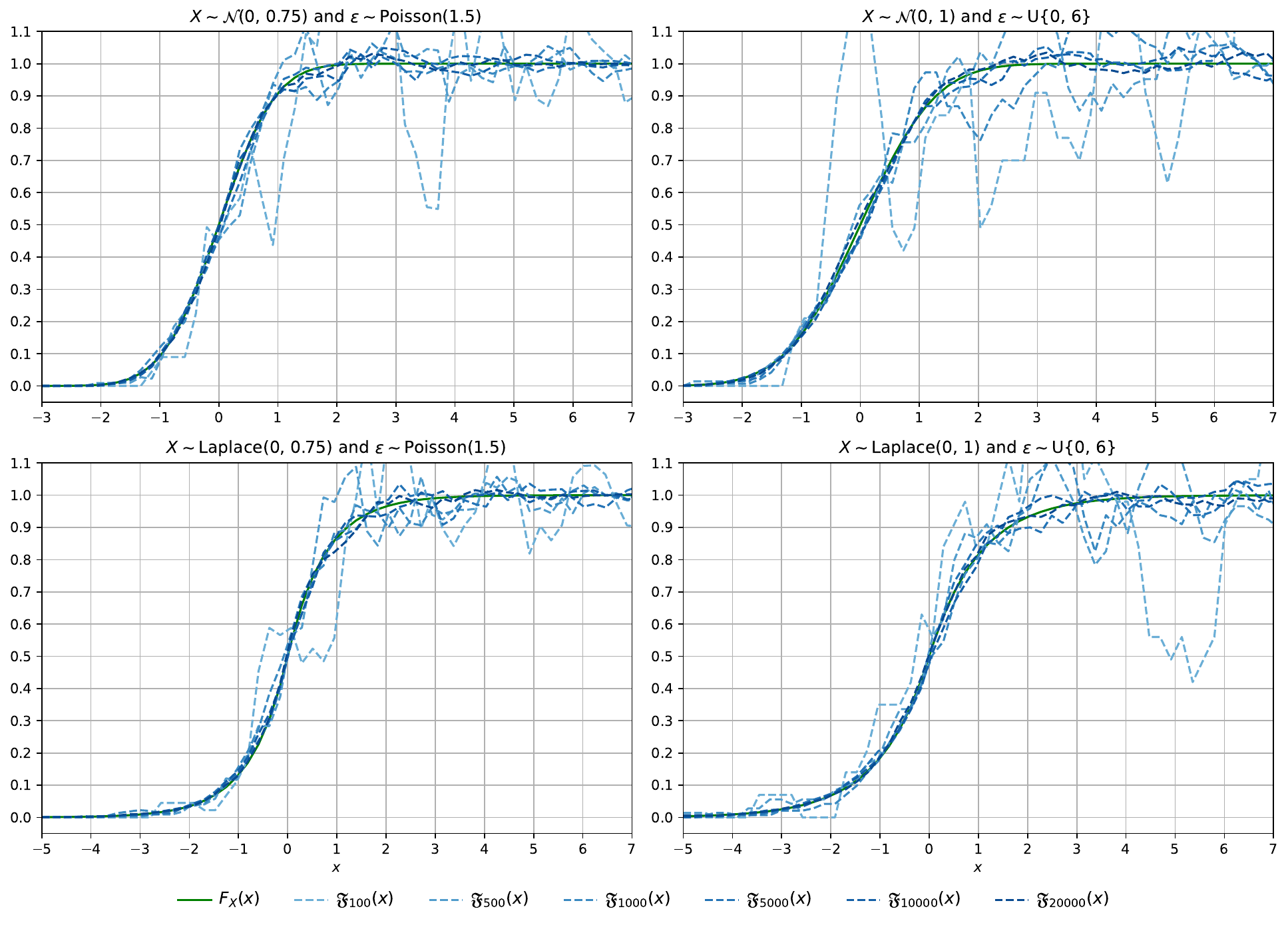}}
\caption{Deconvolution of a bilateral d.f. $F_X$ from a blurred d.f. $F_Y$, through a plug-in estimator $\mathfrak{F}_n$ that was constructed from a $Y$-sample of size $n \in \N$, with the aid of Corollary \ref{BspXbelZdiskr}. Again, $(z_0, t) = (0, 1)$ and hence $P_Y = F_Y$, in all scenarios.}
\label{ExPlotsDecBilatDiscCont}
\end{figure}

\section{A transformation of the convolution equation} \label{SecTransConvEq}

With regard to general setups of errors in variables, the technique from the previous section, aiming at a transformation of a first kind to a second kind convolution equation, requires a slight modification. For this purpose, we first introduce the linear convolution-type operator
\begin{align} \label{ContigRelDecFct5b}
\BigS_\mu\curbr{Q}(\xi) := \intl_{-\infty}^\infty Q(\xi-z) \mu(dz) \hspace{1cm} (\xi\in\R),
\end{align}
where $Q \in L^\infty(\R) \cup L^1(\R)$ and $\mu \in \mathcal{M}(\C, \mathcal{B}(\R))$ is fixed. Clearly, $\BigS_\mu\curbr{Q} \in L^p(\R)$, if $Q \in L^p(\R)$, for each $p \in \curbr{1, \infty}$. Also recall that integration with respect to $\mu$ is equivalent to integration with respect to the d.f. $F_\mu$. Therefore, $\BigS_\mu\curbr{F_\nu} = F_\nu \ast F_\mu$, with commuting factors, if $\nu \in \mathcal{M}(\C, \mathcal{B}(\R))$. In addition,
\begin{align} \label{2025101601}
\BigS_{\mu_2}\curbr{ \BigS_{\mu_1}\curbr{Q} } = \BigS_{\mu_1 \ast \mu_2 }\curbr{Q},
\end{align}
for all $\mu_1, \mu_2 \in \mathcal{M}(\C, \mathcal{B}(\R))$ and $Q \in L^\infty \cup L^1(\R)$. Finally, it is obvious that $\BigS_{\mu_\varepsilon}\curbr{F_X} = F_Y$, $\BigS_{\mu_\varepsilon}\curbr{f_X} = f_Y$ and $\BigS_{\mu_\varepsilon}\curbr{F_X\curbr{\cdot}} = F_Y\curbr{\cdot}$, i.e., the above operator generalizes the convolution equations (\ref{Verteilungsfaltung}), (\ref{Dichtenfaltung}) and (\ref{MassFaltung}). Specifically for the measure
\begin{align} \label{ConnBetwDiffDecFct6}
\pi_\mu := \delta_{ \curbr{0} } - \mu,
\end{align}
which satisfies $\pi_\mu \in \mathcal{M}(\C, \mathcal{B}(\R))$ and has the d.f. $F_{ \pi_\mu } = \ONE_{ \curbr{0 \leq \cdot} } - F_\mu$, we write $\Tau_\mu := \BigS_{ \pi_\mu }$, that is
\begin{align} \label{ConnBetwDiffDecFct5}
\Tau_\mu\curbr{Q}(\xi) =  \intl_\R Q(\xi-z) \pi_\mu(dz) \hspace{1cm} (\xi \in \R).
\end{align}
In particular, because the Dirac measure with mass at the origin corresponds to the identity of convolution of measures, $\BigS_\mu$ and $\Tau_\mu$ are related via the identity
\begin{align} \label{ConnBetwDiffDecFct4}
\Tau_\mu\curbr{Q} = Q - \BigS_\mu\curbr{Q}.
\end{align}
It is hence easy to see that the first kind convolution equation $P := \BigS_{\mu_\varepsilon}\curbr{Q}$, for $Q \in L^\infty(\R) \cup L^1(\R)$, implies that $\BigS_\eta\curbr{P} = \BigS_{\eta \ast \mu_\varepsilon}\curbr{Q}$, for any $\eta \in \mathcal{M}(\C, \mathcal{B}(\R))$, which in turn is equivalent to the second kind convolution equation
\begin{align} \label{2025101001}
Q = \BigS_\eta\curbr{P} + \Tau_{\eta \ast \mu_\varepsilon}\curbr{Q}.
\end{align}
The last equation is even equivalent to the initial equation $P = \BigS_{\mu_\varepsilon}\curbr{Q}$, e.g., if $\eta := \delta_{\curbr{0}}$. At this point, we mention an interesting interpretation of the operator $\Tau_{\eta \ast \mu_\varepsilon}$. In the initial model of errors in variables, the d.f. $F_Y$ is given by the convolution of $F_X$ and $F_\varepsilon$, so that $F_Y = F_X$ if and only if $\mu_\varepsilon = \delta_{ \curbr{0} }$. Consequently, in situations of errors in variables, the Dirac distribution with mass at the origin is associated with the optimal error distribution. On the other side, $\mu_\varepsilon \neq \delta_{\curbr{0}}$ is rather problematic, as then certainly $F_Y \neq F_X$. The function $\Tau_{\eta \ast \mu_\varepsilon}\curbr{F_X}$ represents the deviation of the transformed d.f. $F_\eta \ast F_Y$ from the target d.f. $F_X$, and we infer that the aim of $\eta$ is to induce the best resemblance between these two. Now, the advantage of the above second kind over the initial first kind convolution equation consists in the applicability of Picard's iteration \cite[see, e.g.,][Ch. II]{tricomi1985integral}. Indeed, similar to $\S$\ref{SomeSimpExamp}, this technique once again admits the approximation of the target $Q$ through a so-called Neumann sum. Suppose first, we wish to recover the d.f. $F_X$ from $\BigS_{\mu_\varepsilon}\curbr{F_X} = F_Y$. Then, in view of (\ref{2025101001}), we consider the recursion
\begin{align} \label{IntegralgleichungmitHcrekursiv}
\mathfrak{F}\curbr{\eta}(\cdot,m) := \BigS_\eta\curbr{F_Y} + \Tau_{\eta \ast \mu_\varepsilon}\curbr{\mathfrak{F}\curbr{\eta}(\cdot,m-1)} \hspace{1cm} (m \in \N),
\end{align}
with start function $\mathfrak{F}\curbr{\eta}(\cdot,0) := \BigS_\eta\curbr{F_Y}$. In order to determine a non-recursive form for $\mathfrak{F}\curbr{\eta}(\cdot,m)$, we introduce the Neumann partial sum
\begin{align} \label{2025020203}
\Pi\curbr{\eta}(A, m) := \suml_{\ell=0}^m \pi_{\eta \ast \mu_\varepsilon}^{\ast \ell}(A) \hspace{1cm} ((A, m) \in \mathcal{B}(\R) \times \N_0).
\end{align}
Notice that $\Pi\curbr{\eta}(\cdot, m) \in \mathcal{M}(\C, \mathcal{B}(\R))$, for each $m \in \N_0$, with $\Pi\curbr{\eta}(\cdot, 0) = \delta_{ \curbr{0} }$. Moreover, for convenience, we write $F_{\Pi\curbr{\eta}}(\xi, m) := F_{ \Pi\curbr{\eta}(\cdot, m) }(\xi)$ for the associated d.f.. We can now verify the following result.

\begin{lemma} \label{LemDecFct}
We have
\begin{align} \label{ContigRelDecFct4}
\mathfrak{F}\curbr{\eta}(\xi, m) = (F_\eta \ast F_Y \ast F_{\Pi\curbr{\eta}}(\cdot, m))(\xi) \hspace{1cm} ((\xi, m) \in \R \times \N_0).
\end{align}
\end{lemma}

\begin{proof}
Since $\BigS_\eta\curbr{F_Y} = F_\eta \ast F_Y$, by (\ref{IntegralgleichungmitHcrekursiv}), it is easy to see that $\mathfrak{F}\curbr{\eta}(\cdot, 0) = F_\eta \ast F_Y$ and $\mathfrak{F}\curbr{\eta}(\cdot, 1) = F_\eta \ast F_Y + F_\eta \ast F_Y \ast F_{ \pi_{\eta \ast \mu_\varepsilon} } = F_\eta \ast F_Y \ast F_{\Pi\curbr{\eta}}(\cdot, 1)$. Induction with respect to $m$ thus yields (\ref{ContigRelDecFct4}).
\end{proof}

For completeness, we also mention the case in which one wants to recover the density $f_X$ from $\BigS_{ \mu_\varepsilon }\curbr{f_X} = f_Y$. According to (\ref{2025101001}), we then define the recursion $\mathfrak{f}\curbr{\eta}(\cdot,m) := \BigS_\eta\curbr{f_Y} + \Tau_{\eta \ast \mu_\varepsilon} \curbr{ \mathfrak{f}\curbr{\eta}(\cdot,m-1) }$, for $m \in \N$, with  $\mathfrak{f}\curbr{\eta}(\cdot,0) := \BigS_\eta\curbr{f_Y}$. In computing subsequent iterates, it must be kept in mind that the involved convolutions do not commute, as they are mixtures of complex measures and functions from $L^1(\R)$. Yet, it is easy to confirm the following closed formula.

\begin{lemma}
Assume that $F_X$ or $F_\varepsilon$ is absolutely continuous. Then,
\begin{align} \label{ConnBetwDiffDecFct11}
\mathfrak{f}\curbr{\eta}(\xi, m) = \BigS_{ \eta \ast \Pi\curbr{\eta}(\cdot, m) }\curbr{f_Y}(\xi) \hspace{1cm} ((\xi, m) \in \R \times \N_0).
\end{align}
Especially, the function $\mathfrak{F}\curbr{\eta}(\xi,m)$ from (\ref{ContigRelDecFct4}) is differentiable at Lebesgue almost every $\xi \in \R$, with derivative
\begin{align} \label{ConnBetwDiffDecFct11X}
\frac{d}{d\xi}\mathfrak{F}\curbr{\eta}(\xi,m) = \mathfrak{f}\curbr{\eta}(\xi,m).
\end{align}
\end{lemma}

\begin{proof}
First of all, absolute continuity of $F_X$ or $F_\varepsilon$ implies the existence of $f_Y$. By virtue of (\ref{2025101601}), one then readily verifies (\ref{ConnBetwDiffDecFct11}). The second assertion is a direct consequence of Lebesgue's differentiation theorem.
\end{proof}

We refer to $\mathfrak{F}\curbr{\eta}(\cdot, m)$ and $\mathfrak{f}\curbr{\eta}(\cdot, m)$, respectively, as the \textit{deconvolution function} and \textit{deconvolution density}. The index $m$ corresponds to the accuracy of the approximation for $F_X$ or $f_X$. Beware, however, that $\mathfrak{f}\curbr{\eta}(\cdot, m)$ may exist, although $f_X$ does not. In renewal theory, Neumann partial sums similar to (\ref{2025020203}), but with convolution powers of probability d.fs., are known as renewal functions or renewal measures. Here, appealing to the binomial convolution theorem (Lemma \ref{LemDekmkompakt02}), equivalently,
\begin{align*}
\Pi\curbr{ \eta }(\cdot, m) = \suml_{\ell = 0}^m \suml_{k=0}^\ell \binom{\ell}{k} (-1)^k (\eta \ast \mu_\varepsilon)^{\ast k}.
\end{align*}
In particular, $\pi_{\eta \ast \mu_\varepsilon}^{\ast \ell}$ corresponds to the binomial transform (see Appendix \ref{AppConvIdMeas}) of the sequence $( (\eta \ast \mu_\varepsilon)^{ \ast k} )_{k \in \N_0}$. These convolution powers are discrete measures, if and only if $\eta$ and $\mu_\varepsilon$ are both discrete measures, and are continuous else. In order to make out the behaviour of $\Pi\curbr{ \eta }(\cdot, m)$ and of $\mathfrak{F}\curbr{ \eta }(\cdot, m)$ with respect to $m$, we inevitably need to study the $\ell$-asymptotic behaviour of the binomial transform $\pi_{\eta \ast \mu_\varepsilon}^{\ast \ell}$. It is obvious and also will be confirmed below that this in turn substantially depends on the choice of $\eta$. Generally, when examining convergence of $\Pi\curbr{\eta}(\cdot, m)$ and of $\mathfrak{F}\curbr{\eta}(\cdot, m)$, there are two things to account for. Firstly, of course, the actual existence of the limit. Notice, since $\Pi\curbr{\eta}(\cdot, m)$ at least is a signed measure, that arguments for weak convergence become inapplicable. The second point concerns the identification of the limit. Although $\mathfrak{F}\curbr{ \eta }(\cdot, m)$ is an approximation for the desired target $F_X$, it is not clear if the limit indeed coincides with $F_X$. Our earlier findings, combined with Lemma \ref{LemConvCont20250902}, enable us to easily solve both issues, if $F_\varepsilon$ is associated with a special right-lateral discrete distribution. Denoting by $( \gamma\curbr{\ddot u_+}(z) )_{z \in \Z}$ the sequence from (\ref{2025091103}), the following holds.

\begin{theorem} \label{Theo20250902}
Assume that $\Tbb_\varepsilon \subseteq \curbr{z_0 + tz}_{z \in \N_0}$, for $t > 0$, and $F_\varepsilon\curbr{z_0} > 0$. Define $\lambda_{z_0} := ( F_\varepsilon\curbr{z_0} )^{-1}$ and $\ddot u_{\varepsilon, +}(z) := \delta_{ \curbr{0} }(\curbr{z}) - \lambda_{z_0} F_\varepsilon\curbr{z_0+tz}$. Moreover, for $A \in \mathcal{B}(\R)$, let $m_{t, A} := \max\curbr{ a : a \in \Z \cap t^{-1} A }$, with $m_{t, A} := -\infty$, if $\Z \cap t^{-1}A = \emptyset$. Suppose that $m_{t, A} < \infty$. Then,
\begin{align*}
\sup_{m \in \N_0} |\Pi\curbr{ \delta_{ \curbr{-z_0} } }|(A, m) < \infty,
\end{align*}
and, for each fixed $m_0 \geq m_{t, A}$, we have
\begin{align*}
\liml_{m \rightarrow \infty} \Pi\curbr{ \delta_{ \curbr{-z_0} } }(A, m) &= \lambda_{z_0} \Pi\curbr{ \lambda_{z_0} \delta_{ \curbr{-z_0} } }(A, m_0) \\
&= \lambda_{z_0} \suml_{z = 0}^\infty \gamma\curbr{\ddot u_{\varepsilon, +}}(z) \delta_{ \curbr{tz} }(A).
\end{align*}
\end{theorem}

In view of the convergence properties of $\Pi\curbr{ \delta_{\curbr{-z_0}} }(\cdot, m)$, it is not a surprise, that we can not always expect uniformity with respect to $m \in \N_0$ of the finite total variation on $\R$ of this measure. Indeed, if this would be true, we would also have finite total variation on $\R$ of the limit measure. However, the atoms of this limit are given by $(\gamma\curbr{\ddot u_{\varepsilon, +}}(z))_{z \in \Z}$, and it hence can actually be of infinite total variation on $\R$, according to Example \ref{Ex20251016}. 

\begin{proof}[Proof of Theorem \ref{Theo20250902}]
Under the present assumptions, the measure $\ddot \mu_{\varepsilon-z_0} := \lambda_{z_0} \mu_{\varepsilon-z_0}$ satisfies $\Tbb_{\ddot \mu_{\varepsilon-z_0}^{\ast k}} \subseteq \curbr{tz}_{z \in \N_0}$, with $\ddot\mu_{\varepsilon-z_0}^{\ast k}(A) = \lambda_{z_0}^k \sum_{z=0}^{ m_{t,  A} } F_{\varepsilon-z_0}^{\ast k}\curbr{tz} \delta_{ \curbr{tz} }(A)$, for each $k \in \N_0$. Thus, $( \delta_{ \curbr{ 0 } } - \ddot\mu_{\varepsilon-z_0} )^{\ast \ell}(A) = \sum_{z = 0}^{ m_{t, A} } \ddot u_{\varepsilon, +}^{\ast \ell}(z) \delta_{ \curbr{tz} }(A)$, for all $\ell \in \N_0$. Furthermore, $\ddot u_{\varepsilon, +}(z) = 0$, for each $z \in -\N_0$. Hence, from Lemma \ref{LemDiscConvPwrs}, we deduce that $\ddot u_{\varepsilon, +}^{\ast \ell}(z) = 0$, for all $(\ell, z) \in \N_0^2$, with $z \leq \ell-1$, and eventually also that $( \delta_{ \curbr{ 0 } } - \ddot\mu_{\varepsilon-z_0} )^{\ast \ell}(A) = 0$, whenever $m_{t, A} < 0\leq \ell$ or $\ell > m_{t, A} \geq 0$. Lastly, Lemma \ref{LemConvCont20250902} yields that $\Pi\curbr{ \delta_{ \curbr{-z_0} } }(A, m) = \sum_{\ell=0}^m a_{m, \ell} (\delta_{\curbr{0}} - \ddot \mu_{\varepsilon-z_0} )^{\ast \ell}(A)$, where $0 \leq a_{m, \ell} \leq \lambda_{ z_0 }$, uniformly with respect to $(m, \ell)\in \N_0^2$, and $\lim_{m \rightarrow \infty} a_{m, \ell} = \lambda_{ z_0 }$. Therefore,
\begin{align*}
|\Pi\curbr{ \delta_{ \curbr{-z_0} } }|(A, m) \leq \lambda_{z_0} \sum_{\ell=0}^m |(\delta_{\curbr{0}} - \ddot \mu_{\varepsilon-z_0} )^{\ast \ell}|(A) \leq \lambda_{z_0} \sum_{\ell=0}^{ m_{t, A} } |(\delta_{\curbr{0}} - \ddot \mu_{\varepsilon-z_0} )^{\ast \ell}|(A),
\end{align*}
i.e., $|\Pi\curbr{ \delta_{ \curbr{-z_0} } }|(A, m)$ is uniformly bounded with respect to $m \in \N_0$. With regard to the second part of the theorem, it suffices to suppose that $m_{t, A} \geq 0$, since validity of the asserted identities is obvious for $m_{t,A} < 0$. Now, again from Lemma \ref{LemConvCont20250902} and by dominated convergence, as $m \rightarrow \infty$, we first conclude that the limits of $\Pi\curbr{ \delta_{ \curbr{-z_0} } }(A, m)$ and of $\lambda_{z_0} \Pi\curbr{ \lambda_{z_0} \delta_{ \curbr{-z_0} } }(A,m)$ exist and coincide. In this, for each $m \in \N_0$, we have
\begin{align*}
\Pi\curbr{ \lambda_{z_0} \delta_{\curbr{-z_0}} }(A, m) = \suml_{\ell=0}^m ( \delta_{\curbr{ 0 } } - \ddot \mu_{\varepsilon-z_0} )^{\ast \ell}(A) =  \sum_{\ell=0}^m \suml_{z=0}^{ m_{t, A} } \ddot u_{\varepsilon, +}^{\ast \ell}(z) \delta_{ \curbr{tz} }(A).
\end{align*}
But, again from the cancelling behaviour of the involved convolution powers, we infer that
\begin{align*}
\suml_{\ell=0}^m \suml_{z=0}^{ m_{t, A} } \ddot u_{\varepsilon, +}^{\ast \ell}(z) \delta_{ \curbr{tz} }(A) = \suml_{z=0}^{ m_{t, A} } \suml_{\ell=0}^{ \min\curbr{m, z} } \ddot u_{\varepsilon, +}^{\ast \ell}(z) = \suml_{z=0}^\infty \gamma\curbr{ \ddot u_{\varepsilon, +} }(z) \delta_{ \curbr{tz} }(A),
\end{align*}
whenever $m \geq m_{t, A}$, and the proof is finished.
\end{proof}

As a direct consequence of Theorem \ref{Theo20250902}, we can readily verify convergence of deconvolution function and density for right-lateral $F_X$, and provide a finite representation for their limits.

\begin{corollary}[deconvolution for right-lateral discrete $\varepsilon$] \label{CorFinDecI}
Under the assumptions of Theorem \ref{Theo20250902}, if there exists $\xi_0 \in \R$ with $F_X(\xi_0) = 0$, for each $m_0 \geq \lfloor t^{-1} (\xi-\xi_0) \rfloor$, we have
\begin{align*}
\liml_{m \rightarrow \infty} \mathfrak{F}\curbr{ \delta_{ \curbr{-z_0} } }(\xi, m) = F_X(\xi) = \mathfrak{F}\curbr{ \lambda_{z_0} \delta_{ \curbr{-z_0} } }(\xi, m_0 ) \hspace{1cm} (\xi \in \R).
\end{align*}
Finally, if $F_X$ is absolutely continuous, then also
\begin{align*}
\liml_{m \rightarrow \infty} \mathfrak{f}\curbr{ \delta_{ \curbr{-z_0} } }(\xi, m) = f_X(\xi) = \mathfrak{f}\curbr{ \lambda_{z_0} \delta_{ \curbr{-z_0} } }(\xi, m_0),
\end{align*}
for any $\xi \in \R$, with $\snorm{f_X(\xi-\cdot) \ONE_{t\N_0}}_\infty < \infty$.
\end{corollary}

The last corollary facilitates the unbiased estimation of the d.f. $F_X$ from an i.i.d. sample of $Y$-observations, however, not of the density $f_X$, since it is impossible to estimate $f_Y$ without a bias. Actually, the finite representation for $F_X$ is already included in Corollary \ref{BspXbelZdiskr}. Lastly, we observe that smaller values of the span $t > 0$ force larger choices of the truncation index $m_0$, in order for the approximation to coincide with the target.

\begin{proof}[Proof of Corollary \ref{CorFinDecI}]
Firstly, $\mu_{\varepsilon-z_0}$, $\Pi\curbr{ \delta_{ \curbr{-z_0} } }(\cdot, m)$ and its limit measure are discrete, with atoms in $\curbr{ tz }_{z \in \N_0}$. Secondly, $F_{Y-z_0}(\xi) = 0$, for $\xi \leq \xi_0$. Hence, according to (\ref{ContigRelDecFct4}), it holds that
\begin{align*}
\mathfrak{F}\curbr{ \delta_{ \curbr{-z_0} } }(\xi,m) = \intl_{ (-\infty, \xi-\xi_0 ] } F_{Y-z_0}(\xi - z) \Pi\curbr{ \delta_{ \curbr{-z_0} } }(dz,m) \hspace{1cm} (\xi \in \R).
\end{align*}
In this, since $F_{Y-z_0}$ is bounded and $\sup_{m \in \N_0} |\Pi\curbr{ \delta_{ \curbr{-z_0} } }|((-\infty, \xi-\xi_0 ],m) < \infty$, by Theorem \ref{Theo20250902} and Lemma \ref{Lem2025102701}, the limit as $m \rightarrow \infty$, can be carried out under the integral sign. Thereof, we get $\lim_{m \rightarrow \infty} \mathfrak{F}\curbr{ \delta_{ \curbr{-z_0} } }(\cdot,m) = \lambda_{z_0} (F_{Y-z_0} \ast \Theta\curbr{ \gamma\curbr{\ddot u_{\varepsilon, +}} })$. At the same time, through Lemma \ref{LemRepIndicator}, we identify $F_X = \lambda_{z_0} (F_{Y-z_0} \ast \Theta\curbr{ \gamma\curbr{\ddot u_{\varepsilon, +}} })$, i.e., we obtain the asserted limit for $\mathfrak{F}\curbr{ \delta_{ \curbr{-z_0} } }(\cdot,m)$. The second equality in turn, viz $\mathfrak{F}\curbr{ \lambda_{z_0} \delta_{ \curbr{-z_0} } }(\xi, m_0) = F_X(\xi)$, results from the last part of Theorem \ref{Theo20250902} and the fact that $m_{t, (-\infty, \xi-\xi_0] } = \lfloor t^{-1} (\xi-\xi_0) \rfloor$. Regarding the density $f_X$, we have $f_{Y-z_0}(\xi) = \int_{ (-\infty, \xi-\xi_0] } f_X(\xi-z) F_{\varepsilon-z_0}(dz)$, so that the given condition on $f_X$ implies that also $\snorm{ f_Y(\xi-\cdot) \ONE_{t\N_0} }_\infty < \infty$. The asserted identities for $f_X$ thus can be obtained in the same way as those for $F_X$.
\end{proof}

We confine our discussion on the convergence of the deconvolution function to the above results, and proceed with a brief reconsideration of the case of right-lateral discrete $F_X$ and $F_\varepsilon$, of which $F_X$ has a monotonic support. Then, in Corollary \ref{SatzdiskrDekWkeitsfkt}, we were able to provide a definite representation for $F_X$, solely in terms of quantities that are determined by $Y$ and $\varepsilon$. It was obtained almost in the same fashion as the above function $\mathfrak{F}(\cdot, m)$. Our starting point was the convolution equation for the d.fs., which we convolved with the d.f. of $\eta_\ell := ( F_\varepsilon\curbr{z_0} )^{-1} \delta_{ \curbr{ -\xi_\ell } }$, for $\ell \in \N_0$. The associated probability mass functions at $z_0$ are then related via $( F_\varepsilon\curbr{z_0} )^{-1} F_Y\curbr{z_0+\xi_\ell} = ( F_\varepsilon\curbr{z_0} )^{-1} \int_\R F_\varepsilon\curbr{z_0+\xi_\ell-x}F_X(dx)$. In this equation, we conceived the index $\ell \in \N_0$ as the new argument and eventually performed our transformation to a second kind integral equation, whose finite solution is exactly given in Corollary \ref{SatzdiskrDekWkeitsfkt}. Similarly, for the derivation of Corollary \ref{BspXdiskrZstet}, we convolved $F_Y(z_0) = (F_X \ast F_\varepsilon)(z_0)$ with $\eta_\ell := ( F_\varepsilon\curbr{z_0+\zeta_\ell-\xi_\ell} )^{-1} \delta_{ \curbr{-\zeta_\ell } }$ and again considered the result as a function of $\ell$. Recall that in both of these cases, we first determined the probability mass functions of $X$ and then the d.f.. If we would instead directly approximate $F_X$ by virtue of $\mathfrak{F}\curbr{ \eta }(\cdot, m)$, for a specific $\eta$, we would possibly not receive a finite representation. In fact, in the situation of Corollary \ref{BspXdiskrZstet} with a continuous $F_\varepsilon$, unlike $F_X$, the function $\mathfrak{F}\curbr{\eta}(\cdot, m)$ rather than a step function is continuous for any $\eta$. Altogether, the above discussion has shown that a sophisticated choice of the transforming measure $\eta$ may substantially simplify the solvability of a given deconvolution problem.
\\
\hspace*{1em}The presence of convolution powers as a key component of $\Pi\curbr{\eta}(\cdot, m)$ makes it tempting to put a special focus on convolution semi-groups. Roughly speaking, these are families of probability distributions that are closed under convolution. Accordingly, in these circumstances, $\Pi\curbr{ \eta }(\cdot, m)$ simplifies maybe in the most convenient way. Besides the gamma distribution with fixed scaling parameter, a very important example are stable distributions \cite[][$\S$16.2]{Klenke2020}, such as Cauchy and normal distribution. Since particularly the normal distribution is often considered the most devastating error distribution, we close this section with a short simulation study on such a scenario.

\begin{figure}[h]
\centering
\resizebox*{\textwidth}{!}{\includegraphics{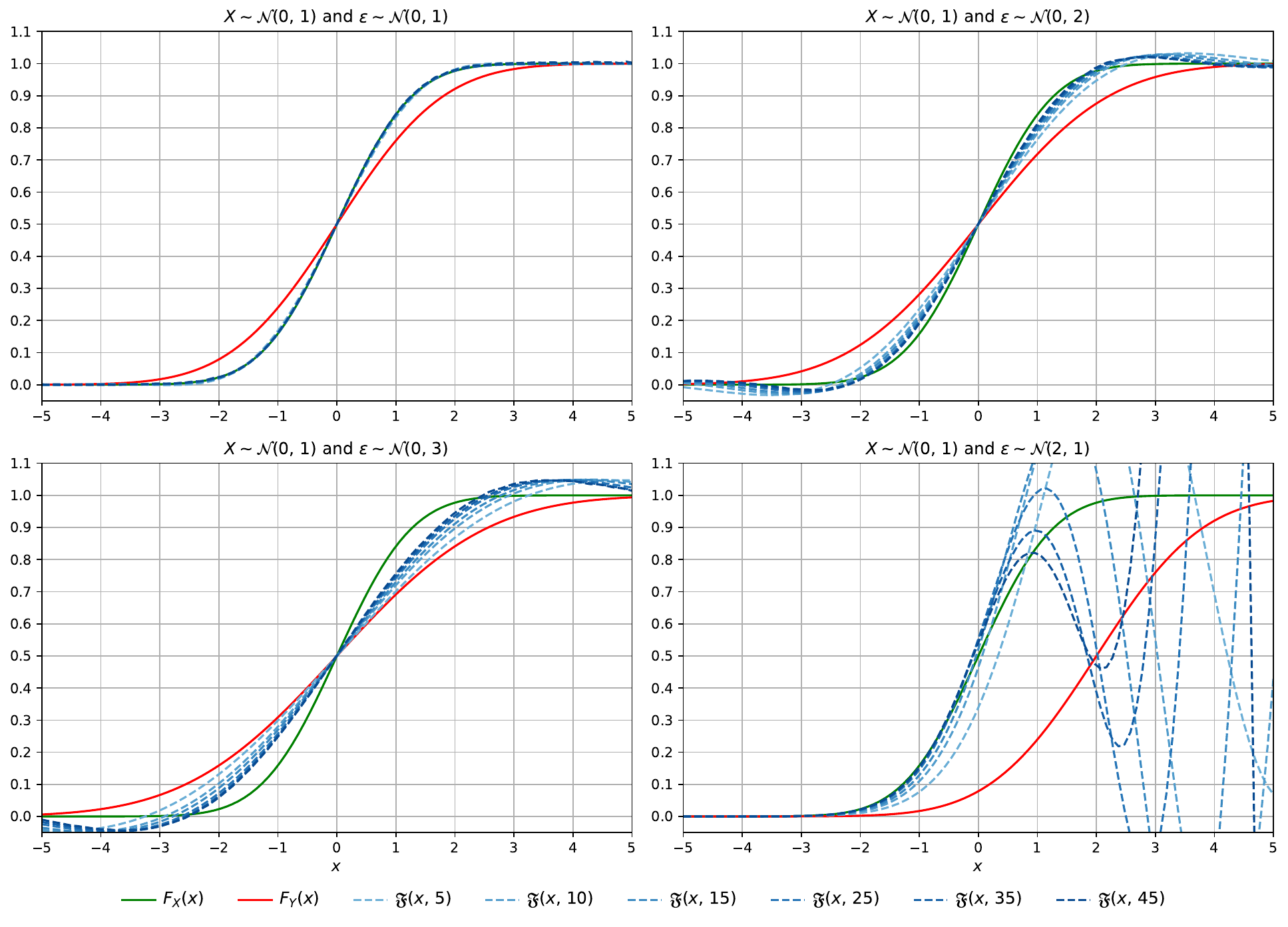}}
\caption{Plots for the deconvolution of a standard normal target from normally distributed errors. Notice how an increasement in the error variance substantially decreases the rate of convergence of $\mathfrak{F}(\cdot, m)$ to $F_X$. Furthermore, in case of non-centered errors, on some segments of the real axis apparently no convergence can be expected.}
\label{PlotsStableDec}
\end{figure}

\begin{example}[normal errors] \label{NormExamp}
Suppose that $\varepsilon \sim \mathcal{N}(c_\varepsilon, \sigma_\varepsilon^2)$, for $c_\varepsilon \in \R$ and $\sigma_\varepsilon > 0$. In these circumstances, $\mu_\varepsilon^{\ast k}$ again corresponds to a normal distribution. As a consequence, writing the Gauss integral in the form $\frac{1}{2}\operatorname{erf}(2^{-\frac{1}{2}}\xi) = (2\pi)^{-\frac{1}{2}} \int_{-\infty}^\xi \exp\curbr{ -\frac{x^2}{2} } dx$, for $\xi \in \R$, where $\operatorname{erf}$ denotes the error function \cite[see][(7.2.1)]{olver2010nist}, the deconvolution function $\mathfrak{F}(\cdot, m) := \mathfrak{F}\curbr{ \delta_{ \curbr{0} } }(\cdot, m)$ admits the representation
\begin{align*}
\mathfrak{F}(\xi, m) = \frac{1}{2} \suml_{\ell=0}^m \suml_{k=0}^\ell \binom{\ell}{k} (-1)^k \intl_{-\infty}^\infty \operatorname{erf}\rb{ \frac{\xi-kc_\varepsilon - y}{\sqrt{2k}\sigma_\varepsilon } } F_Y(dy) \hspace{1cm} ( (\xi, m) \in \R \times \N_0 ).
\end{align*}
If also $X \sim \mathcal{N}(c_X, \sigma_X^2)$, for $\mu_X \in \R$ and $\sigma_X > 0$, the above integral further simplifies. Actually, deconvolution in a completely normal setting is almost trivial, as mean and variance of $X$ directly can be obtained from those of $Y$ and $\varepsilon$. Nevertheless, we chose this example for a first illustration of the properties of $\mathfrak{F}(\cdot, m)$, which are shown in Figure \ref{PlotsStableDec}. These plots suggest that $\mathfrak{F}(\cdot, m)$ converges to $F_X$, as $m \rightarrow \infty$, at least for centered errors. In fact, by Fourier inversion and additional reference to \cite[][$\S$6]{Kaiser04102025}, one can easily confirm that $\sup_{\xi \in \R} | \mathfrak{F}(\xi, m) - F_X(\xi) | = \LandauO{ \curbr{ \curbr{\log m}^{-1} m^{-\sigma_\varepsilon^{-2}} } }$, as $m \rightarrow \infty$, whenever $c_X = c_\varepsilon = 0$ and $\sigma_X = 1$. On the other hand, for $c_\varepsilon\neq 0$, convergence seems to be restricted. We confined our plots to rather small values of $m$, since $m > 45$ inflicts numerical inaccuracies, due to the nature of the binomial coefficient. Finally, given a sample $Y_1, \hdots, Y_n$ of size $n \in \N$, a plug-in estimator for $\mathfrak{F}(\cdot, m)$ is given by
\begin{align*}
\mathfrak{F}_n(\xi, m) := \frac{1}{2n} \suml_{i=1}^n \suml_{\ell=0}^m \suml_{k=0}^\ell \binom{\ell}{k} (-1)^k \operatorname{erf}\rb{ \frac{\xi-kc_\varepsilon-Y_i}{\sqrt{2k} \sigma_\varepsilon } } \hspace{1cm} ( (\xi, m) \in \R \times \N_0 ).
\end{align*}
The performance of $\mathfrak{F}_n(\xi, m)$ is illustrated in Figure \ref{PlotsStableDecEmp}. We focussed on situations with a dominating error variance, since these are expectably more challenging.
\end{example}

\begin{figure}[h]
\centering
\resizebox*{\textwidth}{!}{\includegraphics{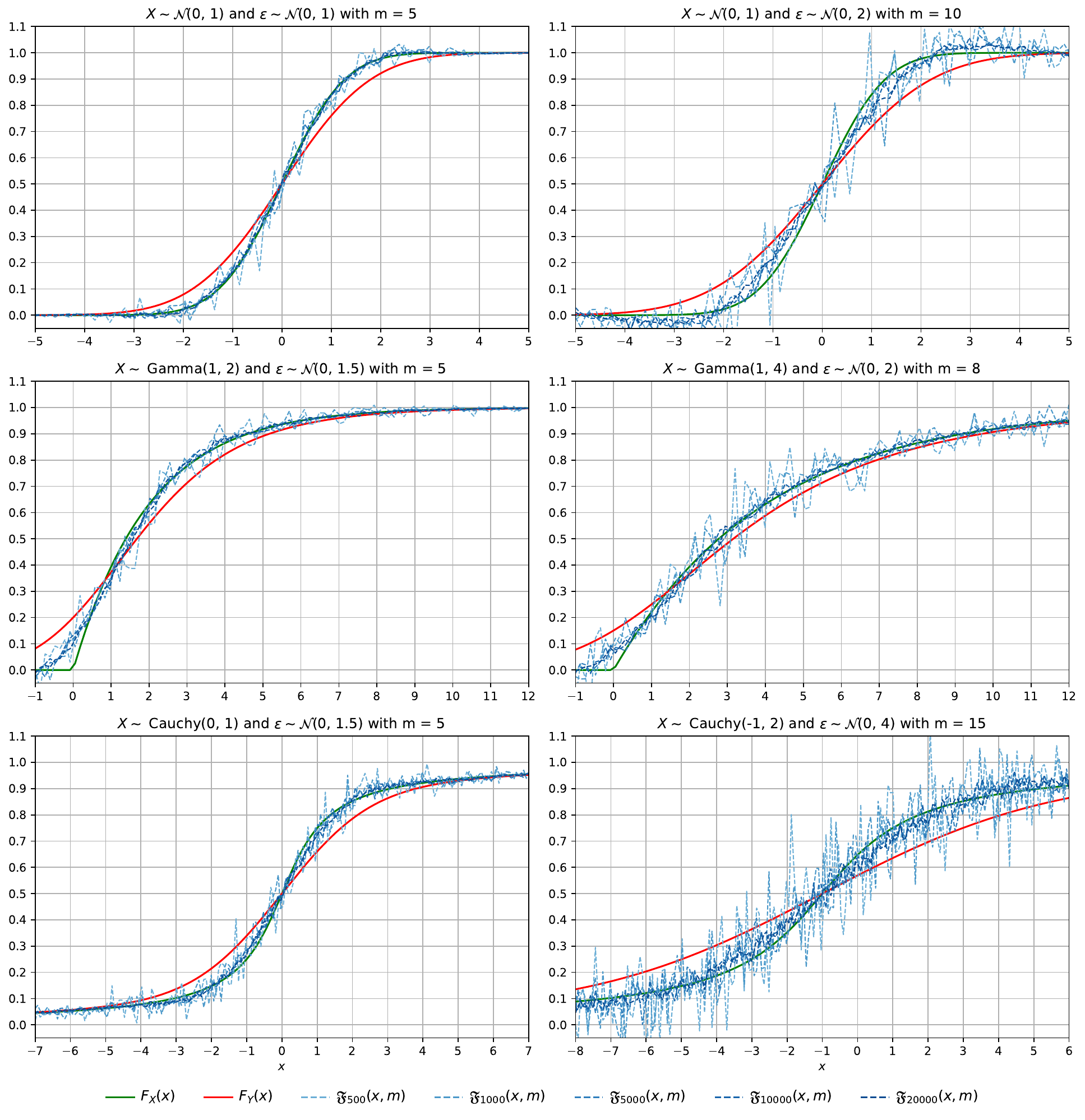}}
\caption{Plots of the plug-in estimator from Example \ref{NormExamp}, in purely normally distributed setups. Recall that $\mathfrak{F}_n(\cdot, m)$ is an estimator for $\mathfrak{F}(\cdot, m)$, which is our approximation for the target $F_X$. We confined to smaller values for $m$, to avoid numerical inaccuracies due to the binomial coefficient.}
\label{PlotsStableDecEmp}
\end{figure}

\section{Invertibility of the convolution operator} \label{SecInvertConvOp}

To complete our study, in the present section, we briefly discuss the invertibility of the convolution operator $\BigS_{\eta \ast \mu_\varepsilon}$ on selected Banach spaces of functions. A well-known criterion for the invertibility of an operator $\Tau \in \mathcal{L}(V)$ on a Banach space $V$ (see \cite[Theorem 10.22]{axler2019measure} or \cite[Lemma 11.16]{robinson2020introduction}) is that the operator norm fulfills
\begin{align} \label{TheoInvOp}
\norm{ \Tau } < 1,
\end{align}
in which case, however, rather than of $\Tau$, invertibility of $\operatorname{Id}_V-\Tau$ follows. More precisely, then, $(\operatorname{Id}_V-\Tau)^{-1} = \sum_{n=0}^\infty \Tau^n$ and $\snorm{(\operatorname{Id}_V-\Tau)^{-1}} \leq (1-\snorm{\Tau})^{-1}$, where $\Tau^n := \Tau \circ \hdots \circ \Tau$ stands for the $n$-times iteration of $\Tau$. Since probability d.fs. and densities, on which we focus in this text, merely form convex sets, we consider the operator $\BigS_{\eta \ast \mu_\varepsilon}$ on the larger spaces of finite signed measures on the Borel $\sigma$-algebra and on the space of absolutely integrable functions on $\R$. On the one hand, for $\nu \in \mathcal{M}(\R, \mathcal{B}(\R))$, we obtain through $\BigS_{\eta \ast \mu_\varepsilon}\curbr{F_\nu}$ the d.f. of the signed measure $\eta \ast \mu_\varepsilon \ast \nu \in \mathcal{M}(\R, \mathcal{B}(\R))$. In particular, $\BigS_{\eta\ast\mu_\varepsilon}\curbr{F_\nu}$ generates the signed measure $\int_A \BigS_{\eta \ast \mu_\varepsilon}\curbr{ F_\nu }(dz) = ( \eta \ast \mu_\varepsilon \ast \nu)(A)$, because integration with respect to a signed measure is equivalent to integration with respect to its d.f.. On the other hand, for $f \in L^1(\R)$, it is obvious that $\BigS_{\eta \ast \mu_\varepsilon}\curbr{f} \in L^1(\R)$. In both of these spaces, according to our earlier observations, the identity operator can be represented as a convolution integral, with integrating measure $\delta_{ \curbr{0} }$. Thus, in view of (\ref{ConnBetwDiffDecFct4}), determining invertibility of $\BigS_{\eta \ast \mu_\varepsilon}$ through the criterion (\ref{TheoInvOp}) amounts to a study of $\BigS_{\pi_{\eta \ast \mu_\varepsilon}} = \Tau_{\eta \ast \mu_\varepsilon}$. But then, the finite counterpart of the series representation for $\BigS_{\eta\ast\mu_\varepsilon}^{-1}$ corresponds to convolution with respect to the measure $\Pi\curbr{ \eta }(\cdot, m)$, which we already encountered in (\ref{2025020203}). Put differently, validity of (\ref{TheoInvOp}) implies the existence of the limit $\Pi\curbr{ \eta }(\cdot, \infty) := \lim_{m \rightarrow \infty} \Pi\curbr{ \eta }(\cdot, m)$, with
\begin{align*}
\BigS_{\eta\ast\mu_\varepsilon}^{-1} = \BigS_{ \Pi\curbr{ \eta }(\cdot, \infty) },
\end{align*}
and specifically the norm convergence of $\Pi\curbr{ \eta }(\cdot, m)$ in the respective Banach space. We begin with the determination of the operator norm of $\Tau_{\eta \ast \mu_\varepsilon}$. It is closely related to the total variation of the underlying signed measure $\pi_{\eta \ast \mu_\varepsilon}$ from (\ref{ConnBetwDiffDecFct6}).

\begin{lemma}[total variation norm] \label{TheoTotVarNorm}
Consider the Banach space $(\mathcal{M}\rb{\R, \mathcal{B}(\R)}, \snorm{\cdot}_{TV})$. Then, $\Tau_{\eta \ast \mu_\varepsilon} \in \mathcal{L}(\mathcal{M}(\R, \mathcal{B}(\R)))$, with
\begin{align*}
\norm{ \Tau_{\eta \ast \mu_\varepsilon} } = \abs{ \pi_{ \eta \ast \mu_\varepsilon } }(\R) \hspace{1cm} (\eta \in \mathcal{M}(\R, \mathcal{B}(\R))).
\end{align*}
\end{lemma}

\begin{proof}
Firstly, for $\mu, \nu \in \mathcal{M}\rb{\R, \mathcal{B}(\R)}$, it is well known that $|\mu \ast \nu|(\R) \leq |\mu|(\R)|\nu|(\R)|$. Secondly, by definition of the operator norm, it holds that
\begin{align*}
\norm{ \Tau_{\eta \ast \mu_\varepsilon} } = \supl\rrb{ \abs{ \pi_{\eta \ast \mu_\varepsilon} \ast \nu } (\R) : \nu \in \mathcal{M}(\R, \mathcal{B}(\R)),~ \abs{\nu}(\R) = 1 }.
\end{align*}
Therefore, $\snorm{\Tau_{\eta \ast \mu_\varepsilon}} \leq \abs{ \pi_{\eta \ast \mu_\varepsilon } } (\R) < \infty$. But $\delta_{\rrb{0}} \in \mathcal{M}(\R,\mathcal{B}(\R))$, with $|\delta_{\rrb{0}}|(\R) = 1$ and $\int_A\Tau_{\eta \ast \mu_\varepsilon}\curbr{ \delta_{ \curbr{0} } }(dz) = \pi_{\eta \ast \mu_\varepsilon}(A)$, for all $A \in \mathcal{B}(\R)$. We conclude that $\snorm{\Tau_{\eta \ast \mu_\varepsilon}\curbr{ \delta_{ \curbr{0} } } }_{TV} = \snorm{\pi_{\eta \ast \mu_\varepsilon} }_{TV}$, which shows that the asserted bound is sharp.
\end{proof}

As the difference of a degenerate and a signed measure, the type of $\pi_{ \eta \ast \mu_\varepsilon }$ usually depends on the properties of $\eta$ and $\mu_\varepsilon$. Specifically the evaluation of the total variation can be very complicated. A useful tool in this context can be the Jordan decomposition \cite[cf.][Theorem 9.30]{axler2019measure}, according to which any signed measure can be represented as the difference of two unique non-negative measures that are singular with respect to each other. The total variation then just equals the sum of the total variations of each measure. However, it is no simple task to find the Jordan decomposition. In particular, despite $\mu_\varepsilon$ is a probability measure, the Jordan decomposition of the convolution of $\mu_\varepsilon$ with the signed measure $\eta$ in general does not coincide with the Jordan decomposition of $\eta$ convolved with $\mu_\varepsilon$. The latter is then just a difference of two measures, however, these are not necessarily singular with respect to each other. As a matter of fact, unfortunately, the total variation of $\pi_{\eta \ast \mu_\varepsilon}$ can not be computed without additional assumptions on $\eta$. In case of continuous ingredients the following holds, due to the fact that the smoothest ingredient always dominates convolution.

\begin{lemma} \label{LemTotVarCont}
Let $\eta \in \mathcal{M}(\R, \mathcal{B}(\R))$. Then, $| \pi_{\eta \ast \mu_\varepsilon} |(\R) \geq 1$, if $\mu_\varepsilon$ or $\eta$ is continuous.
\end{lemma}

\begin{proof}
Since $|\pi_{\eta \ast \mu_\varepsilon}| : \mathcal{B}(\R) \rightarrow [0, \infty)$ is a measure, from $\sigma$-additivity, we infer that $|\pi_{\eta \ast \mu_\varepsilon}|(\R) = |\pi_{\eta \ast \mu_\varepsilon}|(\curbr{0}) + |\pi_{\eta \ast \mu_\varepsilon}|(\R \setminus \curbr{0})$. In this, $\pi_{\eta \ast \mu_\varepsilon}(\curbr{0}) = 1$, because $(\eta \ast \mu_\varepsilon)(\curbr{0}) = 0$, by continuity, and hence also $|\pi_{\eta \ast \mu_\varepsilon}|(\curbr{0}) = 1$.
\end{proof}

We remind the reader that continuity of $\mu_\varepsilon$ is predetermined in applications, and only the structure of $\eta$ can be chosen. Finally, since Lemma \ref{LemTotVarCont} shows that the condition (\ref{TheoInvOp}) can never hold for continuous $\mu_\varepsilon$ or $\eta$, we confine our subsequent discussion to cases, in which both measures have at least one atom. Besides, we assume that $\eta \geq 0$. Consequently, writing $\pi_{\eta \ast \mu_\varepsilon}$, according to (\ref{ConnBetwDiffDecFct6}), in the form
\begin{align} \label{2025071701}
\pi_{\eta \ast \mu_\varepsilon}(A) &= \delta_{\rrb{0}}(A) \rb{ 1 - (\eta \ast \mu_\varepsilon)(\curbr{0}) } - (\eta \ast \mu_\varepsilon)(A\setminus\rrb{0}) \hspace{1cm} (A \in \mathcal{B}(\R)),
\end{align}
directly unfolds the Jordan decomposition, and thereby conveniently facilitates the calculation of the total variation.

\begin{lemma} \label{LemVarDegMeas}
For each $\eta \in \mathcal{M}(\R, \mathcal{B}(\R))$ with $\eta \geq 0$, we have
\begin{align*}
\abs{ \pi_{\eta \ast \mu_\varepsilon} }(\R) =
\begin{cases}
1 + \eta(\R) - 2(F_\eta \ast F_\varepsilon)\curbr{0}, & \mbox{if } 0 \leq (F_\eta \ast F_\varepsilon)\curbr{0} < 1, \\
\eta(\R) - 1 & \mbox{if } (F_\eta \ast F_\varepsilon)\curbr{0} \geq 1.
\end{cases}
\end{align*}
Notice that $(F_\eta \ast F_\varepsilon)\curbr{0} = \sum_{z \in D_{F_\eta} \cap D_{F_\varepsilon}} F_\varepsilon\curbr{z} F_\eta\curbr{-z}$.
\end{lemma}

\begin{proof}
First of all, since $\mu_\varepsilon(\R) = 1$, it is obvious that $(\eta \ast \mu_\varepsilon)(\R) = \eta(\R)$. Moreover, $(\eta \ast \mu_\varepsilon)(\curbr{0}) = (F_\eta \ast F_\varepsilon)\curbr{0}$. Now, by inspection of (\ref{2025071701}), we see that $\pi_{\eta \ast \mu_\varepsilon}$ either is the difference of two non-negative measures that are singular with respect to each other or equals a purely non-positive measure, depending on $(F_\eta \ast F_\varepsilon)\curbr{0}$. The first applies if $0 \leq (F_\eta \ast F_\varepsilon)\curbr{0} < 1$, in which case the Jordan decomposition yields $| \pi_{\eta \ast \mu_\varepsilon} |(\R) = 1 - (F_\eta \ast F_\varepsilon)\curbr{0} + (\eta \ast \mu_\varepsilon)(\R \setminus \curbr{0})$. On the other side, the measure is purely non-positive if $(F_\eta \ast F_\varepsilon)\curbr{0} \geq 1$, in which case $| \pi_{\eta \ast \mu_\varepsilon} |(\R) = | 1 - \eta(\R) |$. In particular, then $\eta(\R) \geq 1$, since $\eta(\R) = (\eta\ast \mu_\varepsilon)(\R) \geq (F_\eta \ast F_\varepsilon)\curbr{0} \geq 1$.
\end{proof}

With the aid of Lemma \ref{LemVarDegMeas}, we can finally establish a sufficient condition for the invertibility of $\BigS_{\eta \ast \mu_\varepsilon}$, given that $\eta$ is a non-negative measure.

\begin{theorem} \label{Theo20250820}
On $(\mathcal{M}\rb{\R, \mathcal{B}(\R)}, \snorm{\cdot}_{TV})$ and $(L^1(\R), \snorm{\cdot}_1)$, the operator $\BigS_{\eta \ast \mu_\varepsilon}$ is invertible, for any $\eta \in \mathcal{M}(\R, \mathcal{B}(\R))$, with $\eta \geq 0$ and
\begin{align*}
\eta(\R) < 2 \min \rrb{ (F_\eta \ast F_\varepsilon)\curbr{0}, 1}.
\end{align*}
\end{theorem}

\begin{proof}
The sufficiency of the given conditions in the space $(\mathcal{M}\rb{\R, \mathcal{B}(\R)}, \snorm{\cdot}_{TV})$ directly follows from (\ref{TheoInvOp}), in view of Lemmas \ref{TheoTotVarNorm} and \ref{LemVarDegMeas}. Furthermore, it is easy to confirm that $\snorm{\Tau_{\eta \ast \mu_\varepsilon}\curbr{q}}_1 \leq\snorm{q}_1\abs{\pi_{\eta \ast \mu_\varepsilon}}(\R)$, for all $q \in L^1(\R)$. Hence, in $(L^1(\R), \snorm{\cdot}_1)$, the operator norm satisfies $\snorm{\Tau_{\eta \ast \mu_\varepsilon}} \leq \abs{\pi_{\eta \ast \mu_\varepsilon}}(\R)$ and $\Tau_{\eta \ast \mu_\varepsilon} \in \mathcal{L}(L^1(\R))$. Furthermore, the condition (\ref{TheoInvOp}) is applicable and holds under the same conditions as in the first part.
\end{proof}

To tie in with our example from Corollary \ref{CorFinDecI}, we briefly consider the measure $\eta := \lambda \delta_{ \curbr{-z_0} }$, for $\lambda > 0$ and $z_0 \in \R$. In these circumstances, $\eta(\R) = \lambda$ and $(F_\eta \ast F_\varepsilon)\curbr{0} = \lambda F_\varepsilon\curbr{z_0}$. Hence, Theorem \ref{Theo20250820} holds, whenever
\begin{align*}
0 \leq \lambda < 2 \min\curbr{ \lambda F_\varepsilon\curbr{z_0}, 1} .
\end{align*}
This is certainly fulfilled by $\lambda := (F_\varepsilon\curbr{z_0})^{-1}$, if $F_\varepsilon\curbr{z_0}> \frac{1}{2}$. Finally, as is pointed out in the introduction to \cite[][$\S$11.5]{robinson2020introduction}, there consists a remarkable difference between an operator \textit{having an inverse} and \textit{being invertible}. Indeed, an operator can have an inverse, although it need not be invertible. The reason is that invertibility is a special property, which implies continuity and boundedness of the inverse operator. Therefore, the inapplicability of the condition (\ref{TheoInvOp}) merely suggests the unboundedness of the inverse operator of $\BigS_{ \eta \ast \mu_\varepsilon }$ on the respective Banach space. Nevertheless, as $m \rightarrow \infty$, the functions $\mathfrak{F}\curbr{\eta}(\cdot, m)$ and $\mathfrak{f}\curbr{\eta}(\cdot, m)$ from (\ref{ContigRelDecFct4}) and (\ref{ConnBetwDiffDecFct11}) may still converge, even in the considered spaces, depending on $F_\eta \ast F_Y$. In fact, the convergence behaviour of the Neumann partial sum $\Pi\curbr{\eta}(\cdot, m)$ may essentially change after additional convolution with $\eta \ast \mu_Y$.

\section{Conclusion and future work}

Altogether, in this text, we proposed various modifications for the initial convolution equations in the additive model of errors in variables. In some cases, these even gave rise to a finite representation of an inverse. However, our discussion is far from complete and leaves many open questions that will be subject of further research. The main question clearly is the behaviour of $\mathfrak{F}\curbr{ \eta }(\cdot, m)$, for a fixed pair of d.fs. $F_X$ and $F_\varepsilon$, under different choices of $\eta$, and the effect of the truncation index $m$. Due to the dominant appearance of convolution powers, in view of the product rule for c.fs. \cite[see][$\S$3.3]{Lukacs1970}, it is tempting to continue further studies in the Fourier domain. Since $\eta \ast \mu_Y \ast \Pi\curbr{\eta}(\cdot, m) \in \mathcal{M}(\C, \mathcal{B}(\R))$, the Fourier-Stieltjes transform $\Phi_{\mathfrak{F}\curbr{\eta}}(t,m) := \int_{-\infty}^\infty e^{itx} \mathfrak{F}\curbr{\eta}(dx,m)$ represents a uniformly continuous function of $t \in \R$, for each $m \in \N_0$. More precisely, according to (\ref{ContigRelDecFct4}), from $\Phi_{ \eta \ast \mu_Y } = \Phi_\eta\Phi_Y$ and $\Phi_{ \pi_{\eta \ast \mu_\varepsilon} } = 1 - \Phi_\eta\Phi_\varepsilon$, we get
\begin{align*}
\Phi_{\mathfrak{F}\curbr{\eta}}(t,m) &= \Phi_\eta(t)\Phi_Y(t) \suml_{\ell=0}^m \rrb{ 1 - \Phi_\eta(t)\Phi_\varepsilon(t) }^\ell \hspace{1cm} ( (t, m) \in \R \times \N_0).
\end{align*}
The right hand side is a geometric sum, $m$ being its truncation index. Therefore,
\begin{align*}
\Phi_{\mathfrak{F}\curbr{\eta}}(t,m) &=
\begin{cases}
\Phi_X(t) \rrb{ 1 - \rb{ 1 - \Phi_\eta(t) \Phi_\varepsilon(t) }^{m+1} }, & \mbox{if } \Phi_\eta(t)\Phi_\varepsilon(t) \neq 0, \\
0, & \mbox{else.}
 \end{cases}
\end{align*}
The obtained representation suggests that an escape to the Fourier domain may also simplify the actual evaluation of the deconvolution quantities, since, in addition to the convolution powers, we avoid the numerically instable binomial coefficient. Yet, there is a hidden pitfall. Considering $t \in \R$ with $\Phi_\eta(t)\Phi_\varepsilon(t) \neq 0$, it shows that $\Phi_{\mathfrak{F}\curbr{\eta}}(t,m)$, as $m \rightarrow \infty$, converges if and only if
\begin{align*}
\abs{ 1 - \Phi_\eta(t) \Phi_\varepsilon(t) } < 1,
\end{align*}
in which event the limit equals $\Phi_X$(t). Thereof, however, simply because $\mathfrak{F}\curbr{ \eta }(\cdot, m)$ is a complex d.f., we may not conclude weak convergence to $F_X$, e.g., as in the continuity theorem \cite[][$\S$3.6]{Lukacs1970}. In fact, apparently there is \textit{no} connection between convergence of $\mathfrak{F}\curbr{\eta}(\cdot, m)$ and of $\Phi_{\mathfrak{F}\curbr{\eta}}(\cdot, m)$. For instance, assume that $\Tbb_X = \N_0$ and $\varepsilon \sim \mbox{Poisson}(2)$. Then, on the one hand, Corollary \ref{CorFinDecI} states that $\lim_{m \rightarrow \infty} \mathfrak{F}\curbr{ \delta_{ \curbr{0} } }(\xi, m) = F_X(\xi)$, even for all $\xi \in \R$, not only at continuity points of $F_X$. On the other hand, $\Phi_{ \delta_{ \curbr{0} } }(t) \Phi_\varepsilon(t) = \Phi_\varepsilon(t) = \exp\curbr{ 2(\exp\curbr{it}-1) }$, with $|1 - \Phi_\varepsilon(0)|=0$ and $|1-\Phi_\varepsilon(5)| \approx 1,10$. Hence, by continuity and periodicity, as $m \rightarrow \infty$, the Fourier-Stieltjes transform $\Phi_{\mathfrak{F}\curbr{ \delta_{ \curbr{0} } }}(\cdot,m)$ converges but also diverges on infinitely many intervals.
\\
\hspace*{1em}Finally, we mention that convergence of $\Phi_{\mathfrak{F}\curbr{ \eta }}(\cdot,m)$ can be achieved in any setting, if we specify $\eta_\varepsilon(A) := \mu_\varepsilon(-A)$, for $A \in \mathcal{B}(\R)$, as the conjugate of the measure $\mu_\varepsilon$. This choice is useful for an approach that entirely relies on Fourier transforms. In fact, the transformed c.f. $\Phi_{\eta_\varepsilon}\Phi_\varepsilon = |\Phi_\varepsilon|^2$ then corresponds to a symmetric distribution. Example \ref{NormExamp} already suggests that symmetry is a beneficial property of errors. Details on this idea, such as convergence in the domain of d.fs. and advantages with regard to estimation of $F_X$, will be discussed in detail elsewhere \cite[see, e.g.,][]{Kaiser2025deconvolutionarbitrarydistributionfunctions}.

\backmatter

\section*{Declarations}

\bmhead{Funding}

The author gratefully acknowledges the support of the WISS 2025 project 'IDA Lab Salzburg' (Land Salzburg, 20102/F2300464-KZP, 20204-WISS/225/348/3-2023).

\begin{appendices}

\section{Convolution identities for complex measures} \label{AppConvIdMeas}

The fact that the convolution of complex measures corresponds to some kind of product, implies various useful identities, of which the following is of most frequent use in this text.

\begin{lemma}[binomial convolution] \label{LemDekmkompakt02}
Consider $j_0 \in \N$, complex measures $\mu_1, \mu_2 : \mathcal{B}(\R) \rightarrow \overline \R + i\overline \R$ and $A \in \mathcal{B}(\R)$, such that $|\mu_t|(A-x\Tbb_{\mu_1}-y\Tbb_{\mu_2}) < \infty$, for all $t \in \curbr{1, 2}$ and $x, y \geq 0$, with $x+y \leq j_0$. Then, for every $0 \leq j \leq j_0$, it holds that $|(\mu_1 + \mu_2)^{\ast j}|(A) < \infty$, with
\begin{align*}
(\mu_1 + \mu_2)^{\ast j}(A) = \suml_{k=0}^j \binom{j}{k} (\mu_1^{\ast (j-k)} \ast \mu_2^{\ast k})(A).
\end{align*}
\end{lemma}

If $A = \curbr{x}$ or $A = (-\infty, x]$, for $x \in \R$, we directly get convolution powers and a binomial theorem for convolutions of sequences or d.fs., respectively. The assumptions are clearly satisfied, whenever $\mu_1, \mu_2 \in \mathcal{M}(\C, \mathcal{B}(\R))$.

\begin{proof}
Under the given assumptions, the binomial sum on the right hand side of the asserted identity defines a complex measure of finite total variation on $A$, for each $0 \leq j \leq j_0$. It therefore suffices to verify the identity, for which we proceed by induction. The cases $j\in\curbr{0, 1}$ are trivial. Assuming that the given identity holds up to the index $j\leq j_0-1$, we next confirm its validity for $j+1$. Elementary manipulations show that
\begin{align*}
(\mu_1 + \mu_2)^{\ast (j+1)}(A) &= ((\mu_1 + \mu_2) \ast (\mu_1 + \mu_2)^{\ast j})(A) \\
&= \suml_{ k = 0 }^j \binom{ j }{ k } \rrb{ (\mu_1^{\ast (j+1-k)} \ast \mu_2^{\ast k})(A) + (\mu_1^{\ast (j-k)} \ast \mu_2^{\ast (k+1)})(A) } \\
&= \mu_1^{\ast (j+1)}(A) + \mu_2^{\ast (j+1)}(A) \\
&\hspace{1cm} + \suml_{ k = 0 }^{j - 1} \rrb{ \binom{ j }{ k } + \binom{ j }{ k + 1 } } (\mu_1^{\ast (j-k)} \ast \mu_2^{\ast (k+1)})(A).
\end{align*}
An application of Pascal's rule \cite[][(26.3.5)]{olver2010nist} eventually finishes the proof.
\end{proof}

Specifically $(\delta_{\curbr{0}} - \mu_2)^{\ast j}$, in view of Lemma \ref{LemDekmkompakt02}, represents the binomial transform of the sequence $( \mu_2^{\ast k} )_{k\in \N_0}$. Generally, the \textit{binomial transform} of a sequence $(p(\ell))_{\ell \in \N_0} \subset \C$ refers to the binomial sum
\begin{align*}
\Beta\curbr{p}(\ell) := \suml_{k=0}^\ell \binom{\ell}{k} (-1)^k p(k) \hspace{1cm} (\ell \in \Z),
\end{align*}
which generates the sequence $( \Beta\curbr{p}(\ell) )_{\ell \in \N_0}$. By repeated application, it turns out that $\Beta\curbr{\Beta\curbr{p}}(\ell) = p(\ell)$, for every $\ell \in \N_0$, i.e., the binomial transform is an involution \cite[see, e.g.,][]{Haukkanen1995}. Our next result highlights various remarkable properties of convolution powers of measures, whose mass is concentrated only on the positive integers. Roughly speaking, these can be written as a weighted sum over the set
\begin{align}
\mathcal{C}_{j, \ell} := \rrb{ (z_1, \hdots, z_j) \in \N^j : \sum_{i=1}^j z_i = \ell } \hspace{1cm} ((j, \ell) \in \N \times \Z),
\end{align}
which is known as the \textit{set of compositions of $\ell$ into $j$ parts} \cite[cf.][Definition 1.9]{Flajolet_Sedgewick_2009} (not to be confused with partitions). In \cite[Example 1.6]{Flajolet_Sedgewick_2009}, it has been shown that $|\mathcal{C}_{j, \ell}|=\binom{\ell-1}{j-1}$, for all $(\ell, j) \in \N^2$.

\begin{lemma} \label{LemDiscConvPwrs}
Let $\mu : \mathcal{B}(\R) \rightarrow \overline \R + i\overline \R$ be a complex measure, with $\Tbb_\mu = \N$ and $|\mu(\curbr{\ell})| < \infty$, for each $\ell \in \N$. Then, $\Tbb_{ \mu^{\ast j} } \subset \N$ and $|\mu^{\ast j}(\curbr{\ell})| < \infty$, for each $(j, \ell) \in \N^2$. In particular, $\mu^{\ast j}(\curbr{\ell}) = 0$, for all $\ell \leq j-1$, and
\begin{align} \label{ProdRepConvPwr}
\mu^{\ast j}( \curbr{ \ell } ) = \suml_{ \vec z \in \mathcal{C}_{j, \ell} } \prodl_{i=1}^j \mu( \curbr{z_i} ) \hspace{1cm} ((j, \ell) \in \N \times \Z),
\end{align}
where $\vec z := (z_1, \hdots, z_j)$. In addition, if there exists $K \in \N$, with $\mu(\curbr{\ell}) = 0$, for all $\ell > K$, then also $\mu^{\ast j}( \curbr{\ell} ) = 0$, for all $\ell > jK$.
\end{lemma}

\begin{proof}
By assumption, $\mu = \sum_{z=1}^\infty \mu(\curbr{z}) \delta_{\curbr{z}}$. It shows that $\mu(\curbr{\ell})=0$, for $\ell \leq 0$, and that (\ref{ProdRepConvPwr}) holds for $j=1$. We proceed by induction, supposing that the first two asserted properties hold up to an index $j$. Then, from the recursion for convolution powers, we get $\mu^{\ast (j+1)}(\curbr{\ell}) = \sum_{z=1}^{\ell-j} \mu^{\ast j}(\curbr{\ell-z}) \mu(\curbr{z})$. Hence, $\mu^{\ast (j+1)}(\curbr{\ell}) = 0$, whenever $\ell \leq j$. Furthermore, an application of (\ref{ProdRepConvPwr}) yields
\begin{align*}
\mu^{\ast (j+1)}(\curbr{\ell}) = \suml_{z=j}^{\ell-1} \mu(\curbr{\ell-z}) \suml_{ \vec z \in \mathcal{C}_{j, z} } \prodl_{i=1}^j \mu( \curbr{z_i} ) = \suml_{ \vec z \in \mathcal{C}_{j+1, \ell} } \prodl_{i=1}^{j+1} \mu( \curbr{z_i} ).
\end{align*}
It remains to verify the third property. This, however, is obvious from (\ref{ProdRepConvPwr}), since at least one factor of each summand equals zero, if $\ell > jK$.
\end{proof}

Our final result essentially relies on the binomial convolution theorem. Its series analogue for $q, q_0 \in \C$, with $|1-q|< 1$, $|q_0-q| < |q_0|$ and $q_0 \neq 0$, is the identity $\sum_{\ell=0}^\infty (1-q)^\ell = q^{-1} = q_0^{-1} \sum_{\ell=0}^\infty (1-q_0^{-1} q)^\ell$, corresponding to the geometric series.

\begin{lemma}[geometric convolution contiguity] \label{LemConvCont20250902}
Assume that the complex measure $\nu : \mathcal{B}(\R) \rightarrow \overline \R + i\overline \R$, for $(m, B) \in \N_0 \times \mathcal{B}(\R)$, satisfies $|\nu^{\ast \ell}|(B) < \infty$, for all $0 \leq \ell \leq m$. Then, defining $\mu := \delta_{\curbr{0}} - \nu_0^{-1} \nu$, for $0 < \nu_0 \leq 1$, it holds that
\begin{align*}
\suml_{\ell=0}^m (\delta_{\curbr{0}} - \nu)^{\ast \ell}(B) = \suml_{\ell = 0}^m a_{m, \ell} \mu^{\ast \ell}(B),
\end{align*}
where $0 \leq a_{m, \ell} \leq \nu_0^{-1}$, uniformly with respect to $(m, \ell) \in \N_0^2$, and $\lim_{m \rightarrow \infty} a_{m, \ell} =  \nu_0^{-1}$, for each $\ell \in \N_0$.
\end{lemma}

\begin{proof}
By virtue of the binomial convolution theorem, i.e., Lemma \ref{LemDekmkompakt02}, we first expand $(\delta_{ \curbr{0} } - \nu)^{\ast \ell}$ and then observe that also $\nu^{ \ast k} = \nu_0^k \sum_{ t = 0 }^k \binom{ k }{ t }(-1)^t \mu^{\ast t}$. Thereby, after additional manipulations, for every $m \in \N_0$, we obtain $\sum_{\ell = 0}^m (\delta_{\curbr{0}} - \nu)^{\ast \ell}(B) = \sum_{t=0}^m a_{m, t} \mu^{\ast t}(B)$, in terms of
\begin{align*}
a_{m, t} := \suml_{\ell=t}^m \suml_{k=t}^\ell \binom{\ell}{k} \binom{k}{t} (-1)^{k-t} \nu_0^k.
\end{align*}
Equivalently, according to the binomial theorem,
\begin{align*}
a_{m, t} = \frac{ \nu_0^t }{ t! } \suml_{n=0}^{m-t} \frac{(n+t)!}{n!}(1- \nu_0)^n \hspace{1cm} (0 \leq t \leq m).
\end{align*}
Note that $(n+t)! = \Gamma(n+t+1)$, where $\Gamma$ refers to the well-known gamma function, i.e., $(n+t)! = \int_0^\infty z^{n+t} e^{-z} dz$. With the aid of this integral representation, since $0 \leq 1 - \nu_0 < 1$, appealing to monotone convergence, we compute
\begin{align*}
\liml_{m \rightarrow \infty} a_{m, t} = \frac{ \nu_0^t }{t!} \intl_0^\infty x^t e^{-\nu_0 x} dx = \nu_0^{-1} \hspace{1cm} (t \in \N_0).
\end{align*}
In particular, $0 \leq a_{m, t} \leq (t!)^{-1} \nu_0^t \int_0^\infty z^t e^{-\nu_0z} dz = \nu_0^{-1}$, uniformly with respect to $(m, t) \in \N_0^2$. The proof is hence completed.
\end{proof}

\section{Convergence of complex measures}

Convergence of finite measures is known as weak convergence. Various criteria to verify this kind of convergence are provided by the well-known Portmanteau theorem. Unfortunately, these can not directly be transferred to complex measures. For that reason, we give a short convergence test specifically for complex measures, which basically generalizes the Helly-Bray theorem in the version of Ch. 1, Theorem 16.4 in \cite{widder1946}.

\begin{lemma} \label{Lem2025102701}
Let $A \in \mathcal{B}(\R)$ and let $\mu_m, \mu : \mathcal{B}(\R) \rightarrow \overline \R + i \overline \R$ be complex measures, for $m \in \N_0$, with $\lim_{m \rightarrow \infty} \mu_m(E) = \mu(E)$, for each $E \subseteq A$, and $\sup_{m \in \N} |\mu_m|(A) < \infty$. Then, $|\mu|(A) < \infty$ and, for every measurable $f : A \rightarrow \C$, we have
\begin{align*}
\liml_{m \rightarrow \infty} \intl_A f(z) \mu_m(dz) = \intl_A f(z) \mu(dz),
\end{align*}
whenever $\snorm{g\ONE_A}_\infty < \infty$, for $g := f \ONE_{\Tbb_\mu \cup \bigcup_{m = 0}^\infty \Tbb_{\mu_m} }$.
\end{lemma}

\begin{proof}
In the sequel, let $K\in \N$ and $(A_k)_{1 \leq k \leq K} \subseteq A$ be arbitrary. First of all, $\sum_{k=1}^K |\mu(A_k)| = \lim \inf_{m \rightarrow \infty} \sum_{k=1}^K |\mu_m(A_k)| \leq \sup_{m \in \N} |\mu_m|(A)$, from which the finite total variation of $\mu$ on $A$ follows. Furthermore, for any simple function $s(z) := \sum_{k=1}^K s_k \ONE_{ A_k }(z)$, with $s_k \neq 0$, the convergence $\int_A s(z) \mu_m(dz) \rightarrow \int_A s(z) \mu(dz)$ is a direct consequence of our assumptions. Finally, by construction,
\begin{align*}
\intl_A f(z) (\mu_m-\mu)(dz) = \intl_A g(z) (\mu_m-\mu)(dz).
\end{align*}
In this, $g$ is measurable and bounded on $A$, from which we conclude \cite[cf.][Theorem 2.89]{axler2019measure} to each $\delta > 0$, the existence of a simple function $s(z)$, with $\snorm{g-s}_\infty < \delta$. Hence,
\begin{align*}
\abs{ \intl_A f(z) (\mu_m-\mu)(dz) } \leq \delta (|\mu_m|(A)+|\mu|(A)) + \abs{ \intl_A s(z) (\mu_m-\mu)(dz) }.
\end{align*}
The first summand can be made arbitrarily small, whereas the second vanishes, as $m \rightarrow \infty$. The proof is thus completed.
\end{proof}

\end{appendices}

\bibliography{References}

\end{document}